\documentclass[12pt,amsfonts]{article}

\def\normo#1{\left\|#1\right\|}

\def\brk#1{\left(#1\right)}

\def\rev#1{\frac{1}{#1}}
\def\half#1{\frac{#1}{2}}
\def\norm#1{\|#1\|}
\def\jb#1{\langle#1\rangle}
\def\wt#1{\widetilde{#1}}
\def\wh#1{\widehat{#1}}

\newcommand{\N}{{\mathbb N}}
\newcommand{\T}{{\mathbb T}}

\newcommand{\R}{{\mathbb R}}
\newcommand{\C}{{\mathbb C}}
\newcommand{\Z}{{\mathbb Z}}
\newcommand{\ft}{{\mathcal{F}}}

\newcommand{\les}{{\lesssim}}
\newcommand{\ges}{{\gtrsim}}

\newcommand{\Sch}{{\mathcal{S}}}

\newcommand{\dr}{{\omega}}

\input epsf
\usepackage{amsthm}
\usepackage{amssymb}
\usepackage{amsmath}
\usepackage{amsfonts}
\input amssym.def

\begin{document}
\title{Global well-posedness and I-method for the fifth-order Korteweg-de Vries equation}
\author{ Wengu Chen$^1$, Zihua Guo$^2$
\\$^1$Institute of Applied Physics
and Computational Mathematics\\P.O.Box 8009, Beijing 100088,
China\\$^2$ LMAM, School of Mathematical Sciences, Peking University\\
Beijing 100871, China}

\date{E-mail:\,chenwg@iapcm.ac.cn,\,\,zihuaguo@math.pku.edu.cn}

\maketitle

\theoremstyle{plain}
  \newtheorem{theorem}[subsection]{Theorem}
  \newtheorem{proposition}[subsection]{Proposition}
  \newtheorem{lemma}[subsection]{Lemma}
  \newtheorem{corollary}[subsection]{Corollary}
  \newtheorem{conjecture}[subsection]{Conjecture}
    \newtheorem*{conj}{Conjecture}
  \newtheorem{problem}[subsection]{Problem}

\theoremstyle{remark}
  \newtheorem{remark}[subsection]{Remark}
  \newtheorem{remarks}[subsection]{Remarks}

\theoremstyle{definition}
  \newtheorem{definition}[subsection]{Definition}

\newtheorem{thm}{Theorem}[section]
\newtheorem{thmA}[thm]{Theorem A}
\newtheorem{cor}[thm]{Corollary}
\newtheorem{prop}[thm]{Proposition}
\newtheorem{define}[thm]{Definition}
\newtheorem{rem}[thm]{Remark}
\newtheorem{example}[thm]{Example}
\def\theequation{\thesection.\arabic{equation}}

\begin{abstract}
We prove that the Kawahara equation is locally well-posed in
$H^{-7/4}$ by using the ideas of $\bar{F}^s$-type space
\cite{GuoKdV}. Next we show it is globally well-posed in $H^s$ for
$s\geq -7/4$ by using the ideas of ``I-method'' \cite{I-method}.
Compared to the KdV equation, Kawahara equation has less symmetries,
such as no invariant scaling transform and not completely
integrable. The new ingredient is that we need to deal with some new
difficulties that are caused by the lack symmetries of this
equation.

\end{abstract}

\section{Introduction}
\setcounter{equation}{0} \label{sec:1}

This paper is mainly concerned with the global well-posedness of the
Cauchy problem for the Kawahara equation
\begin{equation}\label{kawa}
\left\{
\begin{array}{l}
u_t+ \alpha u_{xxx} + \beta u_{xxxxx} + uu_x =0, \quad x,t\in {\R},\\
u(x,0) =u_0(x),
\end{array}
\right.
\end{equation}
where $\alpha$ and $\beta$ are real constants and $\beta\neq 0$. By
a renormalizing of $u$, we may assume $\beta=1$. The fifth-order KdV
type equations arise in modeling gravity-capillary waves on a
shallow layer and magneto-sound propagation in plasmas (see e.g.
\cite{Ka}).

The well-posedness on the fifth-order KdV type equations has
attracted many attentions. Ponce \cite{Ponce} proved an $H^4$ global
well-posedness for the Cauchy problem of the following general
fifth-order KdV equation
\[u_t+u_x+c_1uu_x+c_2u_{xxx}+c_3u_xu_{xx}+c_4uu_{xxx}+c_5u_{xxxxx}=0, \quad x,t\in \R.\]
In \cite{KPV1,KPV2} Kenig, Ponce and Vega studied the following
high-order dispersive equation
\[u_t+\partial_x^{2j+1}u+P(u,\partial_xu,\cdots,\partial_x^{2j}u)=0,\]
where $P$ is a polynomial without constant or linear terms. For the
Kawahara equation \eqref{kawa}, Cui, Deng and Tao \cite{CDT} proved
$H^s$ LWP for $s>-1$, which is later improved by Wang, Cui and Deng
\cite{WCD} to $s\geq -7/5$. Their proofs are based on Kenig, Ponce
and Vega's work \cite{KPVJAMS96}. In \cite{CLMW}, the authors proved
local well-posedness in $H^s$ for $s>-7/4$ by following the ideas of
$[k;\,Z]$-multiplier \cite{Taokz}. Modified Kawahara equation (with
nonlinear terms $u^2u_x$  in \eqref{kawa} instead of $uu_x$) was
also studied, for example see \cite{TC, CLMW}.

The purpose of this paper is to address the following two issues:
one is LWP at $H^{-7/4}$, the other is GWP in $H^s$ for $s<0$. Our
main motivation of this paper is inspired by \cite{CLMW} and
\cite{GuoKdV}. These two problems arise naturally in view of the
results for the Korteweg-de Vries equation. Compared to the KdV
equation, we will encounter a new difficulty. The equation
\eqref{kawa} doesn't have an invariant scaling transform. We will
use the following scaling transform: if $u(x,t)$ is a solution of
\eqref{kawa}, then for $\lambda>0$, $u_\lambda(x,t)=\lambda^4
u(\lambda x, \lambda^5t)$ is a solution to the following equation
\begin{align}\label{scaledkawa}
u_t+ \mu u_{xxx} + u_{xxxxx} + uu_x =0,\, u(x,0)=\phi(x),
\end{align}
where $\mu=\lambda^2\alpha$ and $\phi(x)=\lambda^4u_0(\lambda x)$.
Thus we see from $\norm{\lambda^4u_0(\lambda
x)}_{\dot{H}^s}=\lambda^{s+7/2}\norm{u_0}_{\dot{H}^s}$ that when
$s>-7/2$ we can assume $\norm{\phi}_{H^s}\ll 1$ by taking
$0<\lambda\ll 1$. Since $0<\lambda\leq 1$, heuristically the
equation \eqref{scaledkawa} has a uniform propagation speed in high
frequency. More generally, we study the following equation
\begin{align}\label{eq:dpde}
\partial_t u+Lu+uu_x=0, \quad u(x,0)=u_0(x),
\end{align}
here $L$ is a Fourier multiplier
\[\wh{Lf}(\xi)=-i\dr(\xi)\]
where the symbol $\dr: \R\rightarrow \R$ is an odd function, and
smooth on $\R\setminus\{0\}$. To study the well-posedness for
\eqref{eq:dpde} in the Sobolev space $H^s$, we will see that the
crucial things are related to the dispersive effect of the equation
\eqref{eq:dpde} in high frequency, since $H^s$ spaces have very good
low frequency structure.

\begin{definition}
Assume $\dr: \R\rightarrow \R$ is an odd function, and smooth on
$\R\setminus\{0\}$. For some $\alpha>0$, $\dr$ is said to have
$\alpha$-order dispersive effect at high frequency if for $|\xi|\ges
1$
\[|\partial_\xi^k \omega(\xi)|\sim |\xi|^{\alpha-k}, \
k=1,2;\quad |\partial_\xi^j \omega(\xi)|\les |\xi|^{\alpha-j}, \
j\geq 3.\] Moreover, we denote $\dr\in D_{hi}(\alpha)$.
\end{definition}

For example, the KdV equation corresponds to $\dr=\xi^3$, then $\dr
\in D_{hi}(3)$, for the Kawahara equation \eqref{scaledkawa}
considered in this paper $\dr=\mu \xi^3-\xi^5 \in D_{hi}(5)$
uniformly on $|\mu| \leq 1$. We consider first the L.W.P of
\eqref{kawa} in $H^{-7/4}$. Then it suffices to consider the
equation \eqref{scaledkawa} under the condition
\[|\mu|\leq 1, \quad \norm{\phi}_{H^{-7/4}}\ll 1.\]
We will use the $\bar{F}^s$ type space that was first used recently
by the second author \cite{GuoKdV}. But different from the KdV
equation, here for the local well-posedness a very weak low
frequency structure will work, since the dispersive effect of
\eqref{kawa} is very strong in high frequency. However, in order to
apply the I-method, we will use the same low frequency structure as
the KdV structure. We prove the following
\begin{theorem}\label{thm:lwp}
The Cauchy problem \eqref{kawa} is locally well-posed in $H^{-7/4}$.
\end{theorem}

Next, we will extend the local solution to a global one, using the
ideas of I-method \cite{I-method}. Compared to the KdV equation, the
Kawahara equation has less symmetries. We will use the ideas in
\cite{GuoWang} to estimate the pointwise bounds of the multipliers.

\begin{theorem}\label{thm:gwp}
The Cauchy problem \eqref{kawa} is globally well-posed in $H^{s}$
for $s\geq -7/4$.
\end{theorem}

In the end of this section we give the notations and definitions. In
Section 2 we prove Theorem \ref{thm:lwp}. In Section 3 we give the
modified energy and pointwise multiplier estimates. In Section 4 we
prove Theorem \ref{thm:gwp}. In Section 5 we give an ill-posedness
result.

\noindent {\bf Notation and Definitions.} Throughout this paper we
fix $0<\mu\leq 1$. We will use $C$ and $c$ to denote constants which
are independent of $\mu$ and not necessarily the same at each
occurrence. For $x, y\in \R$, $x\sim y$ means that there exist $C_1,
C_2 > 0$ such that $C_1|x|\leq |y| \leq C_2|x|$. For $f\in \Sch'$ we
denote by $\widehat{f}$ or $\ft (f)$ the Fourier transform of $f$
for both spatial and time variables,
\begin{eqnarray*}
\widehat{f}(\xi, \tau)=\int_{\R^2}e^{-ix \xi}e^{-it \tau}f(x,t)dxdt.
\end{eqnarray*}
We denote  by $\ft_x$ the Fourier transform on spatial variable and
if there is no confusion, we still write $\ft=\ft_x$. Let
$\mathbb{Z}$ and $\mathbb{N}$ be the sets of integers and natural
numbers, respectively. $\Z_+=\N \cup \{0\}$. For $k\in \Z_+$ let
\[{I}_k=\{\xi: |\xi|\in [2^{k-1}, 2^{k+1}]\}, \ k\geq 1; \quad I_0=\{\xi: |\xi|\leq 2\}.\]
Let $\eta_0: \R\rightarrow [0, 1]$ denote an even smooth function
supported in $[-8/5, 8/5]$ and equal to $1$ in $[-5/4, 5/4]$. We
define $\psi(t)=\eta_0(t)$. For $k\in \Z$ let
$\eta_k(\xi)=\eta_0(\xi/2^k)-\eta_0(\xi/2^{k-1})$ if $k\geq 1$ and
$\eta_k(\xi)\equiv 0$ if $k\leq -1$. For $k\in \Z$ let
$\chi_k(\xi)=\eta_0(\xi/2^k)-\eta_0(\xi/2^{k-1})$. Roughly speaking,
$\{\chi_k\}_{k\in \mathbb{Z}}$ is the homogeneous decomposition
function sequence and $\{\eta_k\}_{k\in \mathbb{Z}_+}$ is the
non-homogeneous decomposition function sequence to the frequency
space. For $k\in \Z$ let $P_k$ denote the operator on $L^2(\R)$
defined by
\[
\widehat{P_ku}(\xi)=\eta_k(\xi)\widehat{u}(\xi).
\]
By a slight abuse of notation we also define the operator $P_k$ on
$L^2(\R\times \R)$ by the formula $\ft(P_ku)(\xi,
\tau)=\eta_k(\xi)\ft (u)(\xi, \tau)$. For $l\in \Z$ let
\[
P_{\leq l}=\sum_{k\leq l}P_k, \quad P_{\geq l}=\sum_{k\geq l}P_k.
\]
Thus we see that $P_{\leq 0}=P_0$.

Let
\begin{align}\label{eq:dr}
\dr(\xi)=\mu\xi^3-\xi^5
\end{align}
be dispersion relation associated to equation \eqref{scaledkawa}.
For $\phi\in \Sch'(\R)$, we denote by $W(t)\phi$ the linear solution
of \eqref{scaledkawa} which is defined by
\[
\ft_x(W(t)\phi)(\xi)=\exp[i\dr(\xi)t]\widehat{\phi}(\xi), \  \forall
\ t\in \R.
\]
 We define the Lebesgue spaces $L_{t\in I}^qL_x^p$
and $L_x^pL_{t\in I}^q$ by the norms
\begin{equation}
\norm{f}_{L_{t\in I}^qL_x^p}=\normo{\norm{f}_{L_x^p}}_{L_t^q(I)},
\quad \norm{f}_{L_x^pL_{t\in
I}^q}=\normo{\norm{f}_{L_t^q(I)}}_{L_x^p}.
\end{equation}
If $I=\R$ we simply write $L_{t}^qL_x^p$ and $L_x^pL_{t}^q$. We will
make use of the $X^{s,b}$ norm associated to equation
\eqref{scaledkawa} which is given by
\begin{eqnarray*}
\norm{u}_{X^{s,b}}=\norm{\jb{\tau-\dr(\xi)}^b\jb{\xi}^s\widehat{u}(\xi,\tau)}_{L^2(\R^2)},
\end{eqnarray*}
where $\jb{\cdot}=(1+|\cdot|^2)^{1/2}$. The spaces $X^{s,b}$ turn
out to be very useful in the study of low-regularity theory for the
dispersive equations. These spaces were first used to systematically
study nonlinear dispersive wave problems by Bourgain \cite{Bou} and
developed by Kenig, Ponce and Vega \cite{KPVJAMS96} and Tao
\cite{Taokz}. Klainerman and Machedon \cite{KlMa} used similar ideas
in their study of the nonlinear wave equation.

In applications we usually apply $X^{s,b}$ space for $b$ very close
to $1/2$. In the case $b=1/2$ one has a good substitute-$l^1$ type
$X^{s,b}$ space. For $k\in \Z_+$ we define the dyadic $X^{s,b}$-type
normed spaces $X_k=X_k(\R^2)$,
\begin{eqnarray}
X_k=\left\{f\in L^2(\R^2):
\begin{array}{l}
f(\xi,\tau) \mbox{ is supported in } I_k\times\R \mbox{ and }\\
\norm{f}_{X_k}=\sum_{j=0}^\infty
2^{j/2}\norm{\eta_j(\tau-\dr(\xi))\cdot f}_{L^2}.
\end{array}
\right\}
\end{eqnarray}
Then we define the $l^1$-analogue of $X^{s,b}$ space ${F}^s$ by
\begin{eqnarray}
\norm{u}_{{F}^s}^2=\sum_{k \geq
0}2^{2sk}\norm{\eta_k(\xi)\ft(u)}_{X_k}^2.
\end{eqnarray}
Structures of this kind of spaces were introduced, for instance, in
\cite{Tata}, \cite{IKT} and \cite{In-Ke} for the BO equation. The
space $F^s$ is better than $X^{s,1/2}$ in many situations for
several reasons (see \cite{GuoKdV, GuoWang}). From the definition of
$X_k$, we see that for any $l\in \Z_+$ and $f_k\in X_k$ (see also
\cite{IKT}),
\begin{align}
\sum_{j=0}^\infty 2^{j/2}\normo{\eta_j(\tau-\dr(\xi))\int
|f_k(\xi,\tau')|2^{-l}(1+2^{-l}|\tau-\tau'|)^{-4}d\tau'}_{L^2}\les
\norm{f_k}_{X_k}.
\end{align}
Hence for any $l\in \Z_+$, $t_0\in \R$, $f_k\in X_k$, and $\gamma
\in \Sch(\R)$, then
\begin{align}\label{eq:propertyXk2}
\norm{\ft[\gamma(2^l(t-t_0))\cdot \ft^{-1}f_k]}_{X_k}\les
\norm{f_k}_{X_k}.
\end{align}
In order to avoid some logarithmic divergence, we need to use a
weaker norm for the low frequency
\begin{align*}
\norm{u}_{\bar{X}_0}=\norm{u}_{L_x^2L_t^\infty}.
\end{align*}
Actually a weaker structure will suffice for LWP. However, we will
need this strong structure to extend it. It is easy to see from
Proposition \ref{Xkembedding} that
\begin{eqnarray}\label{eq:xbar0x0}
\norm{\eta_0(t)P_{\leq 0}u}_{\bar{X}_0}\les \norm{P_{\leq
0}u}_{{X}_0}.
\end{eqnarray}
On the other hand, for any $1\leq q\leq \infty$ and $2\leq r\leq
\infty$ we have
\begin{eqnarray}\label{eq:Xbar0Lqr}
\norm{P_{\leq0}u}_{L_{|t|\leq T}^qL_x^r\cap L_x^rL_{|t|\leq
T}^q}\les_T \norm{P_{\leq0}u}_{L_x^2L_{|t|\leq T}^\infty}.
\end{eqnarray}
For $-7/4\leq s\leq 0$, we define the our resolution spaces
\begin{eqnarray*}
\bar{F}^s=\{u\in \Sch'(\R^2):\norm{u}_{\bar{F}^s}^2=\sum_{k \geq
1}2^{2sk}\norm{\eta_k(\xi)\ft(u)}_{X_k}^2+\norm{P_{\leq
0}(u)}_{\bar{X}_0}^2<\infty\}.
\end{eqnarray*}
For $T\geq 0$, we define the time-localized spaces $\bar{F}^{s}(T)$:
\begin{eqnarray}
\norm{u}_{\bar{F}^{s}(T)}=\inf_{w\in \bar{F}^{s}}\{\norm{P_{\leq
0}u}_{L_x^2L_{|t|\leq T}^\infty}+\norm{P_{\geq 1}w}_{\bar{F}^{s}}, \
w(t)=u(t) \mbox{ on } [-T, T]\}.
\end{eqnarray}

Let $a_1, a_2, a_3\in \R$. It will be convenient to define the
quantities $a_{max}\geq a_{med}\geq a_{min}$ to be the maximum,
median, and minimum of $a_1,a_2,a_3$ respectively. Usually we use
$k_1,k_2,k_3$ and $j_1,j_2,j_3$ to denote integers, $N_i=2^{k_i}$
and $L_i=2^{j_i}$ for $i=1,2,3$ to denote dyadic numbers.

\section{L.W.P. at $H^{-7/4}$}

To prove LWP by using $X^{s,b}$-method, the argument is standard.
The first step is to prove a linear estimate, for its proof we refer
the readers to \cite{GuoWang}.

\begin{proposition}[Linear estimates]\label{proplineares}
(a) Assume $s\in \R$ and  $\phi \in H^s$. Then there exists $C>0$
such that
\begin{eqnarray}
\norm{\psi(t)W(t)\phi}_{\bar{F}^s}\leq C\norm{\phi}_{H^{s}}.
\end{eqnarray}

(b) Assume $s\in \R, k\in \Z_+$ and $u$ satisfies
$(i+\tau-\dr(\xi))^{-1}\ft(u)\in X_k$. Then there exists $C>0$ such
that
\begin{eqnarray}
\normo{\ft\left[\psi(t)\int_0^t W(t-s)(u(s))ds\right]}_{X_k}\leq
C\norm{(i+\tau-\dr(\xi))^{-1}\ft(u)}_{X_k}.
\end{eqnarray}
\end{proposition}

Then the remaining task is to show bilinear estimates. We will need
symmetric estimates which will be used to prove bilinear estimates.
For $\xi_1,\xi_2 \in \R$ and $\omega:\R \rightarrow \R$ as in
\eqref{eq:dr} let
\begin{equation}\label{eq:reso}
\Omega(\xi_1,\xi_2)=\omega(\xi_1)+\omega(\xi_2)-\omega(\xi_1+\xi_2).
\end{equation}
This is the resonance function that plays a crucial role in the
bilinear estimate of the $X^{s,b}$-type space. See \cite{Taokz} for
a perspective discussion. For compactly supported nonnegative
functions $f,g,h\in L^2(\R\times \R)$ let
\begin{align*}
J(f,g,h)=\int_{\R^4}f(\xi_1,\mu_1)g(\xi_2,\mu_2)h(\xi_1+\xi_2,\mu_1+\mu_2+\Omega(\xi_1,\xi_2))d\xi_1d\xi_2d\mu_1d\mu_2.
\end{align*}
We will apply to function $f_{k_i,j_i}\in L^2(\R\times \R)$ are
nonnegative functions supported in $[2^{k_i-1},2^{k_i+1}]\times
\widetilde{I}_{j_i}, \ i=1,\ 2,\ 3$. It is easy to see that
$J({f_{k_1,j_1},f_{k_2,j_2},f_{k_3,j_3}})\equiv 0$ unless
\begin{align}\label{didentity}
|k_{med}-k_{max}|\leq 5, \quad 2^{j_{max}}\sim
\max(2^{j_{med}},|\Omega(\xi_1,\xi_2)|).
\end{align}
We give an estimate on the resonance in the following proposition
that follows from the fundamental calculus theorem.

\begin{proposition}\label{propesreso}
Assume $\max(|\xi_1|,|\xi_2|,|\xi_1+\xi_2|)\geq 10$. Then
\[|\Omega(\xi_1,\xi_2)|\sim |\xi|_{max}^4|\xi|_{min},\]
where
\[|\xi|_{max}=\max(|\xi_1|,|\xi_2|,|\xi_1+\xi_2|),\quad |\xi|_{min}=\min(|\xi_1|,|\xi_2|,|\xi_1+\xi_2|).\]
\end{proposition}

In \cite{GuodgBO} the author actually proved the following
proposition, also see the second author's doctoral thesis (P33-34,
\cite{GuoPhD}).

\begin{lemma}\label{lemsymes}
Let $\alpha>1$. Assume $\dr\in D_{hi}(\alpha)$ and $k_i \in \Z$,
$j_i\in \Z_+$, $N_i=2^{k_i},L_i=2^{j_i}$ for $i=1,2,3$. Let
$f_{k_i,j_i}\in L^2(\R\times \R)$ are nonnegative functions
supported in $[2^{k_i-1},2^{k_i+1}]\times \widetilde{I}_{j_i}, \
i=1,\ 2,\ 3$. Then

(a) For any $k_1, k_2, k_3\in \Z$ and $j_1,j_2,j_3\in \Z_+$,
\begin{align}\label{eq:lemsymesa}
J(f_{k_1,j_1},f_{k_2,j_2},f_{k_3,j_3})\leq C
2^{j_{min}/2}2^{k_{min}/2} \prod_{i=1}^3\norm{f_{k_i,j_i}}_{L^2}.
\end{align}

(b) If $N_{min}\ll N_{med}\sim N_{max}$ and $(k_i,j_i)\neq
(k_{min},j_{max})$ for all $i=1,2,3$,
\begin{align}\label{eq:lemsymesb}
J(f_{k_1,j_1},f_{k_2,j_2},f_{k_3,j_3})\leq C
2^{(j_{1}+j_{2}+j_3)/2}2^{-(\alpha-2)k_{max}/2}2^{-(j_i+k_i)/2}\prod_{i=1}^3\norm{f_{k_i,j_i}}_{L^2}.
\end{align}

(c) For any $k_1,k_2,k_3\in \Z$ with $N_{min}\sim N_{med}\sim
N_{max}\gg 1$ and $j_1,j_2,j_3\in \Z_+$
\begin{align}\label{eq:lemsymesc}
J(f_{k_1,j_1},f_{k_2,j_2},f_{k_3,j_3})\leq C
2^{j_{min}/2}2^{j_{med}/4}2^{-(\alpha-2) k_{max}/4}
\prod_{i=1}^3\norm{f_{k_i,j_i}}_{L^2}.
\end{align}
\end{lemma}

In \cite{CLMW} the authors proved a similar results for
$\dr(\xi)=\mu\xi^3-\xi^5$. However, there seems to be some error in
the $high\times high \rightarrow high$ case in their proof. The main
reason is a wrong estimate on the measure of a set. Nevertheless,
this error doesn't change their LWP results for $s>-7/4$ because
this case is not the worst case for the restriction. More
explicitly, we give a counterexample below which shows part (c) in
Lemma \ref{lemsymes} is sharp for $\dr(\xi)=\mu\xi^3-\xi^5$. It
suffices to show that for some $f_{k_1,j_1}, f_{k_2,j_2},
f_{k_3,j_3}$
\begin{align}
J(f_{k_1,j_1},f_{k_2,j_2},f_{k_3,j_3})\geq C
2^{j_{min}/2}2^{j_{med}/4}2^{-3 k_{max}/4}
\prod_{i=1}^3\norm{f_{k_i,j_i}}_{L^2}.
\end{align}
We use the ideas of ``Knapp example" as in the KdV case
\cite{Taokz}. We may assume $L_1\leq L_2\leq L_3$. Let
\begin{align*}
f_1(\xi,\tau)&=1_{|\xi-N_1|\les
N_1^{-3/2}L_2^{1/2}}\cdot 1_{|\tau-\dr(\xi)|\les L_1},\\
f_2(\xi,\tau)&=1_{|\xi-N_1|\les
N_1^{-3/2}L_2^{1/2}}\cdot 1_{|\tau-\dr(\xi)|\les L_2},\\
f_3(\xi,\tau)&=1_{|\xi-2N_1|\les
N_1^{-3/2}L_2^{1/2}}\cdot 1_{|\tau-\frac{\mu
\xi^3}{4}+\frac{\xi^5}{16}|\les L_2}.
\end{align*}
Then we take $f_{k_i,j_i}=f_i$ for $i=1,2,3$. It is easy to see that
$f_i$ satisfy the support properties, and
\[\prod_{i=1}^3\norm{f_{k_i,j_i}}_{L^2}\sim N_1^{-9/4}L_1^{1/2}L_2^{7/4}.\]
On the other hand, by the calculation we get
\[\xi_1^3+\xi_2^3=\frac{(\xi_1+\xi_2)^3}{4}+\frac{3(\xi_1+\xi_2)(\xi_1-\xi_2)^2}{4}\]
and
\[\xi_1^5+\xi_2^5=\frac{(\xi_1+\xi_2)^5}{16}+\frac{5(\xi_1+\xi_2)^3(\xi_1-\xi_2)^2}{8}+\frac{5(\xi_1+\xi_2)
(\xi_1-\xi_2)^4}{16},\] thus it is easy to see that
\begin{align*}
J(f_1,f_2,f_3)\ges&
\int_{\R^4}f_1(\xi_1,\mu_1)f_2(\xi_2,\mu_2)d\xi_1d\xi_2d\mu_1d\mu_2\\
\ges&N_1^{-3}L_1L_2^2\ges 2^{j_{1}/2}2^{j_{2}/4}2^{-3 k_{max}/4}
\prod_{i=1}^3\norm{f_{k_i,j_i}}_{L^2}
\end{align*}
which is as desired.

Next, we prove some dyadic bilinear estimates. It certainly work for
general $\dr \in D_{hi}(\alpha)$ for $\alpha\geq 3$. But for our
purpose we restrict ourselves to the case $\dr=\mu\xi^3-\xi^5$. We
will need the estimates for the linear solution to equation
\eqref{scaledkawa}.
\begin{lemma}\label{lemmakawafree}
Let $I\subset \R$ be an interval with $|I|\les 1$, $k\in \Z_+$ and
$k\geq 10$. Then for all $\phi \in \Sch(\R)$ we have
\begin{align}
\norm{W(t)P_k\phi}_{L_t^qL_x^r}\les& 2^{-3k/q}\norm{\phi}_{L^2},\\
\norm{W(t)P_{\leq k}(\phi)}_{L_x^2L_{t\in I}^\infty}\les& 2^{5k/4}
\norm{\phi}_{L^2},\\
\norm{W(t)P_k\phi}_{L_x^4L_{t}^\infty}\les& 2^{k/4}
\norm{\phi}_{L^{2}},\\
\norm{W(t)P_k\phi}_{L_x^\infty L_{t}^2}\les& 2^{-2k}
\norm{\phi}_{L^{2}},
\end{align}
where $(q,r)$ satisfies $2\leq q,r \leq \infty$ and 2/q=1/2-1/r.
\end{lemma}
\begin{proof}
For the first inequality, see \cite{GPW} and also \cite{CT}, for the
second see \cite{KPVJAMS91}. For the third we use the results in
\cite{KPVIUMJ91}, for the last we use the results in
\cite{KPVJAMS91} by noting that $|\dr'(\xi)|\sim 2^{4k}$ if
$|\xi|\sim 2^{k}$.
\end{proof}

Using the extension lemma in \cite{GuoKdV}, then we get immediately
that

\begin{lemma}\label{Xkembedding}
Let $I\subset \R$ be an interval with $|I|\les 1$, $k\in \Z_+$ and
$k\geq 10$. Then for all $u \in \Sch(\R^2)$ we have
\begin{align}
\norm{P_ku}_{L_t^qL_x^r}\les& 2^{-3k/q}\norm{\wh{P_ku}}_{X_k},\\
\norm{P_{\leq k}u}_{L_x^2L_{t\in I}^\infty}\les& 2^{5k/4}
\norm{\wh{P_{\leq k}u}}_{X_k},\\
\norm{P_ku}_{L_x^4L_{t}^\infty}\les& 2^{k/4}
\norm{\wh{P_ku}}_{X_k},\\
\norm{P_ku}_{L_x^\infty L_{t}^2}\les& 2^{-2k}
\norm{\wh{P_ku}}_{X_k},
\end{align}
where $(q,r)$ satisfies $2\leq q,r \leq \infty$ and 2/q=1/2-1/r.
\end{lemma}

\begin{proposition}[high-low]\label{prophighlow}
(a) If $k\geq 10$, $|k-k_2|\leq 5$, then for any $u,v\in \bar{F}^0$
\begin{align}\label{eq:highlowa}
\norm{(i+\tau-\dr(\xi))^{-1}\eta_k(\xi)i\xi\widehat{P_{\leq
0}u}*\widehat{\psi(t)P_{k_2}v}}_{X_k}\les \norm{{P_{\leq 0}u}}_{
L_x^2L_t^\infty}\norm{\widehat{P_{k_2}v}}_{X_{k_2}}.
\end{align}

(b) If $k\geq 10$, $|k-k_2|\leq 5$ and $1\leq k_1\leq k-9$. Then for
any $u, v\in \bar{F}^0$
\begin{align}\label{eq:highlowb}
\norm{(i+\tau-\dr(\xi))^{-1}\eta_k(\xi)i\xi\widehat{P_{k_1}u}*\widehat{P_{k_2}v}}_{X_k}\les\
k^32^{-7k/2}2^{-k_1}
\norm{\widehat{P_{k_1}u}}_{X_{k_1}}\norm{\widehat{P_{k_2}v}}_{X_{k_2}}.
\end{align}
\end{proposition}

\begin{proof}

For simplicity of notations we assume $k=k_2$. For part (a), it
follows from the definition of $X_k$ that
\begin{align}
\norm{(i+\tau-\dr(\xi))^{-1}\eta_k(\xi)i\xi\widehat{P_0u}*\widehat{\psi(t)P_kv}}_{X_k}\les
2^k\sum_{j\geq
0}2^{-j/2}\norm{\widehat{P_0u}*\widehat{\psi(t)P_{k_2}v}}_{L_{\xi,\tau}^2}.
\end{align}
From Plancherel's equality and Proposition \ref{Xkembedding} we get
\[2^k\norm{\widehat{P_0u}*\widehat{\psi(t)P_{k_2}v}}_{L_{\xi,\tau}^2}\les
2^{k}\norm{{P_0u}}_{L_x^2L_t^\infty }\norm{{P_ku}}_{L_x^\infty
L_t^2}\les 2^{-k}\norm{{P_0u}}_{L_x^2L_t^\infty
}\norm{\widehat{P_{k}v}}_{X_{k}},\]which is part (a) as desired. For
part (b), from the definition we get
\begin{align}\label{eq:highlowb1}
\norm{(i+\tau-\dr(\xi))^{-1}\eta_k(\xi)i\xi\widehat{P_{k_1}u}*\widehat{P_kv}}_{X_k}\les2^k\sum_{j_i\geq
0}2^{-j_3/2}\norm{1_{D_{k,j_3}}\cdot u_{k_1,j_1}*v_{k,j_2}}_2,
\end{align}
where
\begin{eqnarray}\label{eq:decomuv}
u_{k_1,j_1}=\eta_{k_1}(\xi)\eta_{j_1}(\tau-\dr(\xi))\widehat{u},\
v_{k,j_2}=\eta_k(\xi)\eta_{j_2}(\tau-\dr(\xi))\widehat{v}.
\end{eqnarray}
From Proposition \ref{propesreso} and \eqref{didentity} we may
assume $j_{max}\geq 4k+k_1-10$ in the summation on the right-hand
side of \eqref{eq:highlowb1}. We may also assume $j_1,j_2,j_3\leq
10k$, since otherwise we will apply the trivial estimates
\[\norm{1_{D_{k_3,j_3}}\cdot u_{k_1,j_1}*v_{k,j_2}}_2\les 2^{j_{min}/2}2^{k_{min}/2}\norm{u_{k_1,j_1}}_2\norm{u_{k_2,j_2}}_2,\]
then there is a $2^{-5k}$ to spare which suffices to give the bound
\eqref{eq:highlowb}. Thus by applying \eqref{eq:lemsymesb} we get
\begin{align}
&2^k\sum_{j_1,j_2,j_3\geq 0}2^{-j_3/2}\norm{1_{D_{k,j_3}}u_{k_1,j_1}*v_{k,j_2}}_2\nonumber\\
&\les \ 2^k\sum_{j_1,j_2,j_3\geq 0}2^{-j/2}2^{j_{min}/2}2^{-3k/2}2^{-k_1/2}2^{j_{med}/2}\norm{u_{k_1,j_1}}_2\norm{v_{k,j_2}}_2\nonumber\\
&\les \ 2^k\sum_{j_{max}\geq
2k+k_1-10}k^32^{-3k/2}2^{-k_1/2}2^{-j_{max}/2}\norm{\widehat{P_{k_1}u}}_{X_{k_1}}\norm{\widehat{P_kv}}_{X_k}\nonumber\\
&\les \
k^32^{-7k/2}2^{-k_1}\norm{\widehat{P_{k_1}u}}_{X_{k_1}}\norm{\widehat{P_kv}}_{X_k},
\end{align}
which completes the proof of the proposition.
\end{proof}

\begin{proposition}\label{prophhh}
If $k\geq 10$, $|k-k_2|\leq 5$ and $k-9\leq k_1\leq k+10$, then for
any $u,\ v \in F^{-7/4}$
\begin{eqnarray}
\norm{(i+\tau-\dr(\xi))^{-1}\eta_{k_1}(\xi)i\xi\widehat{P_{k}u}*\widehat{P_{k_2}v}}_{X_{k_1}}\les
\  2^{-9k/4}
\norm{\widehat{P_{k}u}}_{X_{k}}\norm{\widehat{P_{k_2}v}}_{X_{k_2}}.
\end{eqnarray}
\end{proposition}

\begin{proof}
As in the proof of Proposition \ref{prophighlow} we assume
$k=k_2=k_1$ and it follows from the definition of $X_{k_1}$ that
\begin{align}\label{eq:highhigh2}
\norm{(i+\tau-\dr(\xi))^{-1}\eta_{k_1}(\xi)i\xi\widehat{P_{k}u}*\widehat{P_kv}}_{X_{k_1}}\les
\ 2^{k_1}\sum_{j_1,j_2,j_3\geq
0}2^{-j_1/2}\norm{1_{D_{k_1,j_1}}u_{k,j_2}*v_{k,j_3}}_2,
\end{align}
where $u_{k,j_1},v_{k,j_2}$ are as in \eqref{eq:decomuv} and we may
assume $j_{max}\geq 5k-20$ and $j_1,j_2,j_3\leq 10k$ in the
summation. Applying \eqref{eq:lemsymesc} we get
\begin{align*}
&2^{k_1}\sum_{j_1,j_2,j_3\geq
0}2^{-j_1/2}\norm{1_{D_{k_1,j_1}}u_{k,j_2}*v_{k,j_3}}_2\\
&\les \big(\sum_{j_1=j_{max}}+\sum_{j_2=j_{max}}+\sum_{j_3=j_{max}}\big)2^{-j_1/2}2^{k/4}2^{j_{min}/2}2^{j_{med}/4}\norm{u_{k,j_2}}_2\norm{v_{k,j_3}}_2\\
&:=I+II+III.
\end{align*}
For the contribution of $I$, since it is easy to get the bound, thus
we omit the details. We only need to bound $II$ in view of the
symmetry. We get that
\begin{align*}
II\les&(\sum_{j_2=j_{max},j_1\leq j_3}+\sum_{j_2=j_{max},j_1\geq
j_3})2^{-j_1/2}2^{k/4}2^{j_{min}/2}2^{j_{med}/4}\norm{u_{k,j_2}}_2\norm{v_{k,j_3}}_2\\
:=&II_1+II_2.
\end{align*}
For the contribution of $II_1$, by summing on $j_1$ we have
\begin{align*}
II_1\les&\sum_{j_2=j_{max},j_1\leq j_3}2^{-j_1/2}2^{k/4}2^{j_{1}/2}2^{j_{3}/4}\norm{u_{k,j_2}}_2\norm{v_{k,j_3}}_2\\
\les&\sum_{j_2\geq 5k-20,j_3\geq
0}2^{k/4}2^{j_{3}/2}\norm{u_{k,j_2}}_2\norm{v_{k,j_3}}_2\les2^{-9k/4}
\norm{\widehat{P_{k}u}}_{X_{k}}\norm{\widehat{P_{k_2}v}}_{X_{k_2}},
\end{align*}
which is acceptable. For the contribution of $II_2$, we have
\begin{align*}
II_2\les&\sum_{j_2=j_{max},j_1\geq j_3}2^{-j_1/2}2^{k/4}2^{j_{3}/2}2^{j_{1}/4}\norm{u_{k,j_2}}_2\norm{v_{k,j_3}}_2\\
\les&2^{-9k/4}
\norm{\widehat{P_{k}u}}_{X_{k}}\norm{\widehat{P_{k_2}v}}_{X_{k_2}}.
\end{align*}
Therefore, we complete the proof of the proposition.
\end{proof}

For the low-low interaction, it is the same as the KdV case
\cite{GuoKdV}.
\begin{proposition}[low-low]\label{proplowlow}
If $0\leq k_1,k_2,k_3\leq 100$, then for any $u,\ v \in F^s$
\begin{align}\label{eq:lowlow}
\norm{(i+\tau-\dr(\xi))^{-1}\eta_{k_1}(\xi)i\xi\widehat{\psi(t)P_{k_2}(u)}*\widehat{P_{k_3}(v)}}_{X_{k_1}}\les
\norm{{P_{k_2}u}}_{L_t^\infty L_x^2}\norm{{P_{k_3}v}}_{L_t^\infty
L_x^2}.
\end{align}
\end{proposition}

Now we consider the high-high interactions. This is the only case
where the restriction comes from.

\begin{proposition}[high-high]\label{prophighhigh}
If $k\geq 10$, $|k-k_2|\leq 5$ and $1\leq k_1\leq k-9$, then for any
$u,\ v \in F^0$
\begin{align}\label{eq:highhigh}
\norm{(i+\tau-\dr(\xi))^{-1}\eta_{k_1}(\xi)i\xi\widehat{P_{k}u}*\widehat{P_{k_2}v}}_{X_{k_1}}\les
(2^{-7k/2}+k2^{-4k}2^{k_1/2})
\norm{\widehat{P_{k}u}}_{X_{k}}\norm{\widehat{P_{k_2}v}}_{X_{k_2}}.
\end{align}
\end{proposition}

\begin{proof}
We assume $k=k_2$ and it follows from the definition of $X_{k_1}$
that
\begin{align}\label{eq:highhigh2}
\norm{(i+\tau-\dr(\xi))^{-1}\eta_{k_1}(\xi)i\xi\widehat{P_{k}u}*\widehat{P_kv}}_{X_{k_1}}\les
\ 2^{k_1}\sum_{j_1,j_2,j_3\geq
0}2^{-j_1/2}\norm{1_{D_{k_1,j_1}}u_{k,j_2}*v_{k,j_3}}_2,
\end{align}
where $u_{k,j_2},v_{k,j_3}$ are as in \eqref{eq:decomuv}. For the
same reasons as in the proof of Proposition \ref{prophighlow} we may
assume $j_{max}\geq 4k+k_1-10$ and $j_1,j_2,j_3\leq 10k$. We will
bound the right-hand side of \eqref{eq:highhigh2} case by case. The
first case is that $j_1=j_{max}$ in the summation. Then we apply
\eqref{eq:lemsymesb} and get that
\begin{align*}
&2^{k_1}\sum_{j_1,j_2,j_3\geq
0}2^{-j_1/2}\norm{1_{D_{k_1,j_1}}u_{k,j_2}*v_{k,j_3}}_2\nonumber\\
&\les \ 2^{k_1}\sum_{j_{1}\geq 4k+k_1-10}\sum_{j_2,j_3\geq
0}2^{-j_1/2}2^{-3k/2}2^{{-k_1}/2}2^{(j_{2}+j_3)/2}\norm{u_{k,j_2}}_{2}\norm{v_{k,j_3}}_{2}\\
&\les
2^{-7k/2}\norm{\widehat{P_{k}u}}_{X_{k}}\norm{\widehat{P_{k_2}v}}_{X_{k_2}},
\end{align*}
which is acceptable. If $j_2=j_{max}$, then in this case we have
better estimate for the characterization multiplier. By applying
\eqref{eq:lemsymesb} we get
\begin{align*}
& 2^{k_1}\sum_{j_1,j_2,j_3\geq
0}2^{-j_1/2}\norm{1_{D_{k_1,j_1}}u_{k,j_2}*v_{k,j_3}}_2\nonumber\\
&\les \ 2^{k_1}\sum_{j_{2}\geq 4k+k_1-10}\sum_{j_1\leq 10k,j_3\geq
0}2^{-j_1/2}2^{-2k}2^{(j_{1}+j_3)/2}\norm{u_{k,j_2}}_{2}\norm{v_{k,j_3}}_{2}\\
&\les k
2^{-4k}2^{k_1/2}\norm{\widehat{P_{k}u}}_{X_{k}}\norm{\widehat{P_{k_2}v}}_{X_{k_2}},
\end{align*}
where in the last inequality we use $j_1\leq 10k$. The last case
$j_3=j_{max}$ is identical to the case $j_2=j_{max}$ from symmetry.
Therefore, we complete the proof of the proposition.
\end{proof}

For $k_1=0$ we can prove a similar proposition but with $k2^{-7k/2}$
instead of $2^{-7k/2}$ on the right-hand side of
\eqref{eq:highhigh}. In order to avoid the logarithmic divergence,
we prove the following

\begin{proposition}[$\bar{X}_0$ estimate]\label{propXbares}
Let $|k_1-k_2|\leq 5$ and $k_1\geq 10$. Then we have for all $u,v\in
\bar{F}^0$
\begin{align*}
\normo{\psi(t)\int_0^t W(t-s)P_{\leq
0}\partial_x[P_{k_1}u(s)P_{k_2}v(s)]ds}_{L_x^2 L_t^\infty}\les
2^{-\half{7k_1}}\norm{\wh{P_{k_1}u}}_{X_{k_1}}\norm{\wh{P_{k_2}u}}_{X_{k_2}}.
\end{align*}
\end{proposition}

\begin{proof}
Denote $Q(u,v)=\psi(t)\int_0^t W(t-s)P_{\leq
0}\partial_x[P_{k_1}u(s)P_{k_2}v(s)]ds$. By straightforward
computations we get
\begin{align*}
\ft\left[Q(u,v)\right](\xi,\tau)=&C\int_\R
\frac{\widehat{\psi}(\tau-\tau')-\widehat{\psi}(\tau-\dr(\xi))}{\tau'-\dr(\xi)}\eta_0(\xi)i\xi
\\
&\times\
d\tau'\int_{\xi=\xi_1+\xi_2,\tau'=\tau_1+\tau_2}\wh{P_{k_1}u}(\xi_1,\tau_1)\wh{P_{k_2}v}(\xi_2,\tau_2).
\end{align*}
Fixing $\xi \in \R$, we decompose the hyperplane $\Gamma:=
\{\xi=\xi_1+\xi_2,\tau'=\tau_1+\tau_2\}$ as following
\begin{align*}
\Gamma_1=&\{ |\xi|\les 2^{-4k_1}\} \cap \Gamma;\\
\Gamma_2=&\{ |\xi|\gg 2^{-4k_1},
|\tau_i-\dr(\xi_i)|\ll 3\cdot 2^{4k_1}|\xi|, i=1,2\} \cap \Gamma;\\
\Gamma_3=&\{ |\xi|\gg 2^{-4k_1},
|\tau_1-\dr(\xi_1)|\ges 3\cdot 2^{4k_1}|\xi|\} \cap \Gamma;\\
\Gamma_4=&\{ |\xi|\gg 2^{-4k_1}, |\tau_2-\dr(\xi_2)|\ges 3\cdot
2^{2k_1}|\xi|\} \cap \Gamma.
\end{align*}
Then we get
\[\ft\left[\psi(t)\cdot
\int_0^tW(t-s)P_{\leq
0}\partial_x[P_{k_1}u(s)P_{k_2}v(s)]ds\right](\xi,\tau)=A_1+A_2+A_3+A_4,\]
where
\[
A_i=C\int_\R
\frac{\widehat{\psi}(\tau-\tau')-\widehat{\psi}(\tau-\dr(\xi))}{\tau'-\dr(\xi)}\eta_0(\xi)i\xi
\int_{\Gamma_i}\wh{P_{k_1}u}(\xi_1,\tau_1)\wh{P_{k_2}v}(\xi_2,\tau_2)d\tau'.
\]

We consider first the contribution of the term $A_1$. Using
Proposition \ref{Xkembedding} and Proposition \ref{proplineares}
(b), we get
\[\norm{\ft^{-1}(A_1)}_{ L_x^2L_t^\infty}\les \normo{(i+\tau'-\dr(\xi))^{-1}\eta_0(\xi)i\xi
\int_{\Gamma_1}\wh{P_{k_1}u}(\xi_1,\tau_1)\wh{P_{k_2}v}(\xi_2,\tau_2)}_{X_0}.\]
Since in the area $A_1$ we have $|\xi|\les 2^{-4k_1}$, thus we get
\begin{align*}
&\normo{(i+\tau'-\dr(\xi))^{-1}\eta_0(\xi)i\xi
\int_{A_1}\wh{P_{k_1}u}(\xi_1,\tau_1)\wh{P_{k_2}v}(\xi_2,\tau_2)}_{X_0}\\
&\les \sum_{k_3\leq -4k_1+10}\sum_{j_3\geq 0}
2^{-j_3/2}2^{k_3}\sum_{j_1\geq 0, j_2\geq
0}\norm{1_{D_{k_3,j_3}}\cdot u_{k_1,j_1}*v_{k_2,j_2}}_{L^2}
\end{align*}
where
\[
u_{k_1,j_1}(\xi,\tau)=\eta_{k_1}(\xi)\eta_{j_1}(\tau-\dr(\xi))\wh{u}(\xi,\tau),
v_{k_1,j_1}(\xi,\tau)=\eta_{k_1}(\xi)\eta_{j_1}(\tau-\dr(\xi))\wh{v}(\xi,\tau).
\]
Using \eqref{eq:lemsymesa}, then we get
\begin{align*}
\norm{\ft^{-1}(A_1)}_{ L_x^2L_t^\infty}\les& \sum_{k_3\leq
-4k_1+10}\sum_{j_i\geq 0} 2^{-j_3/2}2^{k_3}2^{j_{min}/2}2^{k_3/2}
\norm{u_{k_1,j_1}}_{L^2}\norm{v_{k_2,j_2}}_{L^2}\\
\les&
2^{-6k_1}\norm{\wh{P_{k_1}u}}_{X_{k_1}}\norm{\wh{P_{k_2}u}}_{X_{k_2}},
\end{align*}
which suffices to give the bound for the term $A_1$.

Next we consider the contribution of the term $A_3$. As for the term
$A_1$, using Proposition \ref{Xkembedding} and Proposition
\ref{proplineares} (b), we get
\begin{align*}
\norm{\ft^{-1}(A_3)}_{ L_x^2L_t^\infty}\les&
\normo{(i+\tau'-\dr(\xi))^{-1}\eta_0(\xi)i\xi \int_{\Gamma_3}\wh{P_{k_1}u}(\xi_1,\tau_1)\wh{P_{k_2}v}(\xi_2,\tau_2)}_{X_0}\\
\les&\sum_{k_3\leq 0}\sum_{j_3\geq 0} 2^{-j_3/2}2^{k_3}\sum_{j_1\geq
0, j_2\geq 0}\norm{1_{D_{k_3,j_3}}\cdot
u_{k_1,j_1}*v_{k_2,j_2}}_{L^2}
\end{align*}
Clearly we may assume $j_3\leq 10k_1$ in the summation above. Using
\eqref{eq:lemsymesb}, then we get
\begin{align*}
\norm{\ft^{-1}(A_3)}_{ L_x^2L_t^\infty}\les& \sum_{k_3\leq
0}\sum_{j_1\geq k_3+4k_1-10, j_2,j_3\geq 0}
2^{k_3}2^{j_{2}/2}2^{-3k_1}
\norm{u_{k_1,j_1}}_{L^2}\norm{v_{k_2,j_2}}_{L^2}\\
\les&
k_12^{-5k_1}\norm{\wh{P_{k_1}u}}_{X_{k_1}}\norm{\wh{P_{k_2}u}}_{X_{k_2}},
\end{align*}
which suffices to give the bound for the term $A_3$. From symmetry,
the bound for the term $A_4$ is the same as $A_3$.

Now we consider the contribution of the term $A_2$. From the proof
of the dyadic bilinear estimates, we know this term is the main
contribution. By computation we get
\begin{align*}
\ft_t^{-1}(A_2)=\psi(t)\int_0^t e^{i(t-s)\dr(\xi)}\eta_0(\xi)i\xi
\int_{\R^2}e^{is(\tau_1+\tau_2)}
\int_{\xi=\xi_1+\xi_2}{u_{k_1}}(\xi_1,\tau_1){v_{k_2}}(\xi_2,\tau_2)\
d\tau_1 d\tau_2ds
\end{align*}
where
\begin{align*}
u_{k_1}(\xi_1,\tau_1)=&\eta_{k_1}(\xi_1)1_{\{|\tau_1-\dr(\xi_1)|\ll
3\cdot 2^{4k_1}|\xi|\}}\wh{u}(\xi_1,\tau_1),\\
v_{k_2}(\xi_2,\tau_2)=&\eta_{k_2}(\xi_2)1_{\{|\tau_2-\dr(\xi_2)|\ll
3\cdot 2^{4k_1}|\xi|\}}\wh{v}(\xi_2,\tau_2).
\end{align*}
By a change of variable $\tau_1'=\tau_1-\dr(\xi_1)$,
$\tau_2'=\tau_2-\dr(\xi_2)$, we get
\begin{align*}
\ft_t^{-1}(A_2)=&\psi(t)e^{it\dr(\xi)}\eta_0(\xi)\xi\int_{\R^2}e^{it(\tau_1+\tau_2)}\int_{\xi=\xi_1+\xi_2}\frac{e^{it(\dr(\xi_1)+\dr(\xi_2)-\dr(\xi))}-e^{-it(\tau_1+\tau_2)}}
{\tau_1+\tau_2-\dr(\xi)+\dr(\xi_1)+\dr(\xi_2)}\\
&\times \
{u_{k_1}}(\xi_1,\tau_1+\dr(\xi_1)){v_{k_2}}(\xi_2,\tau_2+\dr(\xi_2))\
d\tau_1 d\tau_2\\
=&\ft_t^{-1}(I)-\ft_t^{-1}(II).
\end{align*}
For the contribution of the term $II$, we have
\begin{align*}
&\ft_t^{-1}(II)=\int_{\R^2}\psi(t)e^{it\dr(\xi)}\eta_0(\xi)\xi\int_{\xi=\xi_1+\xi_2}
\frac{{u_{k_1}}(\xi_1,\tau_1+\dr(\xi_1)){v_{k_2}}(\xi_2,\tau_2+\dr(\xi_2))}{\tau_1+\tau_2-\dr(\xi)+\dr(\xi_1)+\dr(\xi_2)}\
d\tau_1 d\tau_2.
\end{align*}
Since in the support of $u_{k_1}$ and  $u_{k_2}$ we have
$|\tau_1+\tau_2-\dr(\xi)+\dr(\xi_1)+\dr(\xi_2)|\sim 2^{4k_1}|\xi|$,
then we get from Lemma \ref{lemmakawafree} that
\begin{align*}
\norm{\ft^{-1}(II)}_{L_x^2L_t^\infty}\les&
\int_{\R^2}\normo{\int_{\xi=\xi_1+\xi_2}\xi\frac{{u_{k_1}}(\xi_1,\tau_1+\dr(\xi_1)){v_{k_2}}(\xi_2,\tau_2+\dr(\xi_2))}{\tau_1+\tau_2-\dr(\xi)+\dr(\xi_1)+\dr(\xi_2)}}_{L_\xi^2}d\tau_1d\tau_2\\
\les&2^{-\half{7k_1}}\norm{\wh{P_{k_1}u}}_{X_{k_1}}\norm{\wh{P_{k_2}u}}_{X_{k_2}}.
\end{align*}

To prove the proposition, it remains to prove the following
\begin{eqnarray*}
\norm{\ft^{-1}(I)}_{L_x^2L_t^\infty}\les
2^{-7k_1/2}\norm{\wh{P_{k_1}u}}_{X_{k_1}}\norm{\wh{P_{k_2}u}}_{X_{k_2}}.
\end{eqnarray*}
Compare the term $I$ with the following term $I'$:
\begin{align*}
\ft_t^{-1}(I')
=&\psi(t)e^{it\dr(\xi)}\eta_0(\xi)\xi\int_{\R^2}e^{it(\tau_1+\tau_2)}\int_{\xi=\xi_1+\xi_2}
\frac{e^{it(\dr(\xi_1)+\dr(\xi_2)-\dr(\xi))}}{-\dr(\xi)+\dr(\xi_1)+\dr(\xi_2)}\\
&\times \
{u_{k_1}}(\xi_1,\tau_1+\dr(\xi_1)){v_{k_2}}(\xi_2,\tau_2+\dr(\xi_2))\
d\tau_1 d\tau_2.
\end{align*}
Since on the hyperplane $\xi=\xi_1+\xi_2$ one has
\[-\dr(\xi+\xi)+\dr(\xi_1)+\dr(\xi_2)=\xi_1\xi_2\xi(\xi_1^2+\xi_2^2+\xi^2-\lambda
\alpha)=C\xi_1\xi_2\xi(-2\xi_1\xi_2+2\xi^2-\lambda \alpha).\] In the
integral area, we have $|2\xi^2-\lambda \alpha|\ll |\xi_1\xi_2|$,
thus we get
\[\frac{1}{-2\xi_1\xi_2+2\xi^2-\lambda \alpha}=\frac{1}{-2\xi_1\xi_2}\sum_{n=0}^\infty\brk{\frac{2\xi^2-\lambda\alpha}{2\xi_1\xi_2}}^n\]
Inserting this into $I'$ we have
\begin{align*}
\ft_t^{-1}(I')
=&\psi(t)\eta_0(\xi)\sum_{n=0}^\infty\int_{\R^2}e^{it(\tau_1+\tau_2)}\int_{\xi=\xi_1+\xi_2}
e^{it(\dr(\xi_1)+\dr(\xi_2))}\frac{(2\xi^2-\lambda \alpha)^n}{(\xi_1\xi_2)^{n+2}}\\
&\times \
{u_{k_1}}(\xi_1,\tau_1+\dr(\xi_1)){v_{k_2}}(\xi_2,\tau_2+\dr(\xi_2))\
d\tau_1 d\tau_2.
\end{align*}
Since it is easy to see that (actually we need a smooth version of
$1_{\{|\xi|\gg \lambda\}}$): $\forall \ \lambda>0$,
\[\norm{\ft_x^{-1}1_{\{|\xi|\gg \lambda\}}\ft_x u}_{L_x^2L_t^\infty}\les \norm{u}_{L_x^2L_t^\infty},\]
and setting
\[\ft(f_{\tau_1})(\xi)=\wh{P_{k_1}u}(\xi,\tau_1+\dr(\xi)),\
\ft(g_{\tau_2})(\xi)=\wh{P_{k_2}v}(\xi,\tau_2+\dr(\xi)),\] thus we
get from Lemma \ref{lemmakawafree} that we set
\begin{align*}
\norm{\ft^{-1}(I')}_{L_x^2L_t^\infty}\les&\sum_{n=0}^\infty C^n
\int_{\R^2}\norm{W(t)\partial_x^{-{(n+2)}}f_{\tau_1}W(t)\partial_x^{-{(n+2)}}g_{\tau_2}}_{L_x^2L_t^\infty}d\tau_1d\tau_2\\
\les&\sum_{n=0}^\infty C^n\int_{\R^2}\norm{W(t)\partial_x^{-{(n+2)}}f_{\tau_1}}_{L_x^4L_t^\infty}\norm{W(t)\partial_x^{-{(n+2)}}g_{\tau_2}}_{L_x^4L_t^\infty}d\tau_1d\tau_2\\
\les&2^{-\half{7k_1}}\norm{\wh{P_{k_1}u}}_{X_{k_1}}\norm{\wh{P_{k_2}u}}_{X_{k_2}},
\end{align*}
which gives the bound for the term $II'_1$.

To prove the proposition, it remains to prove the following
\begin{align*}
\norm{\ft^{-1}(I-I')}_{L_x^2L_t^\infty}\les
2^{-7k_1/2}\norm{\wh{P_{k_1}u}}_{X_{k_1}}\norm{\wh{P_{k_2}u}}_{X_{k_2}}.
\end{align*}
Since in the integral area we have $|\tau_i|\ll 2^{4k_1}|\xi|$,
$i=1,2$, thus on the hyperplane $\xi=\xi_1+\xi_2$ we have
\begin{align*}
&\frac{1}{\tau_1+\tau_2-\dr(\xi)+\dr(\xi_1)+\dr(\xi_2)}-\frac{1}{-\dr(\xi)+\dr(\xi_1)+\dr(\xi_2)}\\
=&\sum_{n=1}^\infty
\rev{-\dr(\xi)+\dr(\xi_1)+\dr(\xi_2)}\brk{\frac{\tau_1+\tau_2}{-\dr(\xi)+\dr(\xi_1)+\dr(\xi_2)}}^n\\
=&C\sum_{n=1}^\infty\frac{1}{(\xi_1\xi_2)^2\xi}\sum_{k=0}^\infty
\brk{\frac{2\xi^2-\lambda\alpha}{2\xi_1\xi_2}}^k
\brk{\frac{\tau_1+\tau_2}{(\xi_1\xi_2)^2\xi}}^n \sum_{j_1,\cdots,
j_n=0}^\infty
\prod_{i=1}^n\brk{\frac{2\xi^2-\lambda\alpha}{2\xi_1\xi_2}}^{j_i}.
\end{align*}
The purpose of decomposing this is to make the variable separately,
thus then we can apply Lemma \ref{lemmakawafree}. Then by
decomposing low frequency we get
\begin{align*}
&\ft_t^{-1}(I-I')=\sum_{n=1}^\infty\psi(t)\eta_0(\xi)\int_{\R^2}e^{it(\tau_1+\tau_2)}\sum_{2^{k_3}\gg
2^{-4k_1}\max(|\tau_1|,|\tau_2|)}\chi_{k_3}(\xi)\\
&\times\int_{\xi=\xi_1+\xi_2}e^{it(\dr(\xi_1)+\dr(\xi_2))}
{u_{k_1}}(\xi_1,\tau_1+\dr(\xi_1)){v_{k_2}}(\xi_2,\tau_2+\dr(\xi_2))\frac{1}{(\xi_1\xi_2)^2}\\
&\times \sum_{k=0}^\infty
\brk{\frac{2\xi^2-\lambda\alpha}{2\xi_1\xi_2}}^k
\brk{\frac{\tau_1+\tau_2}{(\xi_1\xi_2)^2\xi}}^n \sum_{j_1,\cdots,
j_n=0}^\infty
\prod_{i=1}^n\brk{\frac{2\xi^2-\lambda\alpha}{2\xi_1\xi_2}}^{j_i}
d\tau_1 d\tau_2.
\end{align*}
Using the fact that $\chi_{k_3}(\xi)(\xi/2^{k_3})^{-n}$ is a
multiplier for the space $L_x^2L_t^\infty$ and as for the term $I'$,
we get
\begin{align*}
&\norm{\ft^{-1}(I-I')}_{L_x^2L_t^\infty}\\
\les &\sum_{n=1}^\infty \int_{\R^2} \sum_{2^{k_3}\gg
2^{-4k_1}\max(|\tau_1|,|\tau_2|)} C^n
|\tau_1+\tau_2|^n2^{-nk_3}2^{-4nk_1}\\
&\times \
2^{-7k_1/2}\norm{\ft(f_{\tau_1})}_{L^2}\norm{\ft(g_{\tau_2})}_{L^2}d\tau_1
d\tau_2\\
\les&
2^{-7k_1/2}\norm{\wh{P_{k_1}u}}_{X_{k_1}}\norm{\wh{P_{k_2}u}}_{X_{k_2}}.
\end{align*}
Therefore, we complete the proof of the proposition.
\end{proof}

For $u,v\in \bar{F}^s$ we define the bilinear operator
\begin{align}
B(u,v)=\psi(\frac{t}{4})\int_0^tW(t-\tau)\partial_x\big(\psi^2(\tau)u(\tau)\cdot
v(\tau)\big)d\tau.
\end{align}
In order to apply a fixed point argument, all the issues are then
reduced to show the boundness of $B:\bar{F}^s\times
\bar{F}^s\rightarrow \bar{F}^s$. Then Theorem \ref{thm:lwp} follows
from standard arguments.

\begin{proposition}[Bilinear estimates]\label{propbilinearbd}
Assume $-7/4\leq s\leq 0$. Then there exists $C>0$ such that
\begin{eqnarray}\label{eq:bilinearbd}
\norm{B(u,v)}_{\bar{F}^s}\leq
C(\norm{u}_{\bar{F}^s}\norm{v}_{\bar{F}^{-7/4}}+\norm{u}_{\bar{F}^{-7/4}}\norm{v}_{\bar{F}^s})
\end{eqnarray}
hold for any $u,v\in \bar{F}^s$.
\end{proposition}
\begin{proof}
The proof is similar to the Proposition 4.2 \cite{GuoKdV}. We omit
the details.
\end{proof}

\section{Modified energies}

In this section we follow I-method \cite{I-method} to extend the
local solution. Let $m: \R^k \rightarrow \C$ be a function. We say
$m$ is symmetric if $m(\xi_1,\cdots, \xi_k)=m(\sigma(\xi_1,\cdots,
\xi_k))$ for all $\sigma \in S_k$, the group of all permutations on
$k$ objects. The symmetrization of $m$ is the function
\begin{equation}
[m]_{sym}(\xi_1,\xi_2,\cdots, \xi_k)=\rev{k!}\sum_{\sigma\in
S_k}m(\sigma(\xi_1,\xi_2,\cdots,\xi_k)).
\end{equation}
We define a $k-linear$ functional associated to the function $m$
(multiplier) acting on $k$ functions $u_1,\cdots,u_k$,
\begin{equation}
\Lambda_k(m;u_1,\cdots,u_k)=\int_{\xi_1+\cdots+\xi_k=0}m(\xi_1,\cdots,\xi_k)\widehat{u_1}(\xi_1)\cdots
\widehat{u_k}(\xi_k).
\end{equation}
We will often apply $\Lambda_k$ to $k$ copies of the same function
$u$. $\Lambda_k(m;u,\ldots,u)$ may simply be written $\Lambda_k(m)$.
By the symmetry of the measure on hyperplane, we have
$\Lambda_k(m)=\Lambda_k([m]_{sym})$. The following proposition may
be directly verified by using the Kawahara equation
\eqref{scaledkawa}.
\begin{proposition}
Suppose $u$ satisfies the Kawahara equation \eqref{scaledkawa} and
that $m$ is a symmetric function. Then
\begin{equation}\label{eq:menergy}
\frac{d}{dt}\Lambda_k(m)=\Lambda_k(mh_k)-\Lambda_k(mv_k)-i\half k
\Lambda_{k+1}(m(\xi_1,\ldots,\xi_{k-1},\xi_k+\xi_{k+1})(\xi_k+\xi_{k+1})),
\end{equation}
where
\[h_k=i\mu(\xi_1^3+\xi_2^3+\cdots+\xi_k^3), \quad v_k=i(\xi_1^5+\xi_2^5+\cdots+\xi_k^5).\]
\end{proposition}

We then follow the I-method \cite{I-method} to define a set of
modified energies. Let $m:\R\rightarrow \R$ be an arbitrary even
$\R$-valued function and define the operator by
\begin{align*}
\widehat{If}(\xi)=m(\xi)\widehat{f}(\xi),
\end{align*}
where the multiplier $m(\xi)$ is smooth, monotone, and of the form
for $N\gg 1$
\begin{align}\label{eq:m}
m(\xi)=\left\{
\begin{array}{r}
1, \quad \quad |\xi|<N,\\
N^{-s}|\xi|^s,\quad  |\xi|>2N.
\end{array}
\right.
\end{align}
We define the modified energy $E_I^2(t)$ by
\begin{align*}
E_I^2(t)=\norm{Iu(t)}_{L^2}^2.
\end{align*}
Using Plancherel's identity and that $m$ and $u$ are $\R$-valued,
and $m$ is even, we get
\[E_I^2(t)=\Lambda_2(m(\xi_1)m(\xi_2)).\]
Using \eqref{eq:menergy} then we have
\begin{align*}
\frac{d}{dt}E_I^2(t)=\Lambda_2(m(\xi_1)m(\xi_2)h_2)-\Lambda_2(m(\xi_1)m(\xi_2)v_2)-i\Lambda_3(m(\xi_1)m(\xi_2+\xi_3)(\xi_2+\xi_3)).
\end{align*}
The first two terms vanish. We symmetrize the third term to get
\begin{align*}
\frac{d}{dt}E_I^2(t)=\Lambda_3(-i[m(\xi_1)m(\xi_2+\xi_3)(\xi_2+\xi_3)]_{sym}).
\end{align*}
Let us denote
\begin{align*}
M_3(\xi_1,\xi_2,\xi_3)=-i[m(\xi_1)m(\xi_2+\xi_3)(\xi_2+\xi_3)]_{sym}.
\end{align*}
Form the new modified energy
\[E_I^3(t)=E_I^2(t)+\Lambda_3(\sigma_3)\]
where the symmetric function $\sigma_3$ will be chosen momentarily
to achieve a cancellation. Applying \eqref{eq:menergy} gives
\begin{align}\label{eq:E3}
\frac{d}{dt}E_I^3(t)=\Lambda_3(M_3)+\Lambda_3(\sigma_3h_3)-\Lambda_3(\sigma_3v_3)-\half
3 i\Lambda_4(\sigma_3(\xi_1,\xi_2,\xi_3+\xi_4)(\xi_3+\xi_4)).
\end{align}
Compared to the KdV case \cite{I-method}, there is one more term to
cancel, so we choose
\begin{align*}
\sigma_3=-\frac{M_3}{h_3-v_3}
\end{align*}
to force the three $\Lambda_3$ terms in \eqref{eq:E3} to cancel.
Hence if we denote
\begin{align*}
M_4(\xi_1,\xi_2,\xi_3,\xi_4)=-i\half
3[\sigma_3(\xi_1,\xi_2,\xi_3+\xi_4)(\xi_3+\xi_4)]_{sym}
\end{align*}
then
\begin{align*}
\frac{d}{dt}E_I^3(t)=\Lambda_4(M_4).
\end{align*}
Similarly defining
\[E_I^4(t)=E_I^3(t)+\Lambda_4(\sigma_4)\]
with
\begin{align*}
\sigma_4=-\frac{M_4}{h_4-v_4},
\end{align*}
we obtain
\begin{align*}
\frac{d}{dt}E_I^4(t)=\Lambda_5(M_5)
\end{align*}
where
\begin{align*}
M_5(\xi_1,\ldots,\xi_5)=-2i[\sigma_4(\xi_1,\xi_2,\xi_3,\xi_4+\xi_5)(\xi_4+\xi_5)]_{sym}.
\end{align*}

In order to prove the pointwise estimates for the multiplier
$\sigma_3,\sigma_4$, we need the following lemma which is crucial.
It just follows from simple calculations, thus we do not give the
proof.

\begin{lemma}\label{lem:decreso} (a) Assume $\xi_1+\xi_2+\xi_3+\xi_4=0$, then
\[\xi_1^5+\xi_2^5+\xi_3^5+\xi_4^5=-\frac{5}{2}(\xi_1+\xi_2)(\xi_1+\xi_3)(\xi_2+\xi_3)({\xi_1^2+\xi_2^2+\xi_3^2}+\xi_4^2),\]
and
\[\xi_1^3+\xi_2^3+\xi_3^3+\xi_4^3=-3(\xi_1+\xi_2)(\xi_1+\xi_3)(\xi_2+\xi_3).\]

(b) Assume $\xi_1+\xi_2+\xi_3=0$, then
\[\xi_1^5+\xi_2^5+\xi_3^5=\frac{5}{2}\xi_1\xi_2\xi_3(\xi_1^2+\xi_2^2+\xi_3^2),\]
and
\[\xi_1^3+\xi_2^3+\xi_3^3=3\xi_1\xi_2\xi_3.\]
\end{lemma}

Now we turn to give the pointwise estimates of the multiplier. It is
easy to see that if $m$ is of the form \eqref{eq:m}, then $m^2$
satisfies
\begin{eqnarray}\label{eq:eImul}
\left\{
\begin{array}{l}
m^2(\xi)\sim m^2(\xi') \mbox{ for } |\xi|\sim|\xi'|,\\
(m^2)'(\xi)=O(\frac{m^2(\xi)}{|\xi|}),\\
(m^2)''(\xi)=O(\frac{m^2(\xi)}{|\xi|^2}).
\end{array}\right.
\end{eqnarray}
We will need two mean value formulas which follow immediately from
the fundamental theorem of calculus. If $|\eta|,|\lambda|\ll |\xi|$,
then we have
\begin{equation}\label{eq:mvt}
|a(\xi+\eta)-a(\xi)|\les |\eta|\sup_{|\xi'|\sim |\xi|}|a'(\xi')|,
\end{equation}
and the double mean value formula that
\begin{equation}\label{eq:dmvt}
|a(\xi+\eta+\lambda)-a(\xi+\eta)-a(\xi+\lambda)+a(\xi)|\les
|\eta||\lambda|\sup_{|\xi'|\sim |\xi|}|a''(\xi')|.
\end{equation}
In order to use the formulas, we extend the surface supported
multiplier $\sigma_3$ to the whole space as in \cite{KochTataru}.

\begin{proposition}
If $m$ is of the form \eqref{eq:m}, then for each dyadic
$\lambda\leq \eta$ there is an extension of $\sigma_3$ from the
diagonal set
\[\{(\xi_1,\xi_2,\xi_3)\in \Gamma_3(\R), |\xi_1|\sim \lambda,\quad |\xi_2|, |\xi_3|\sim \eta\}\]
to the full dyadic set
\[\{(\xi_1,\xi_2,\xi_3)\in \R^3, |\xi_1|\sim \lambda,\quad |\xi_2|, |\xi_3|\sim \eta\}\]
which satisfies
\begin{equation}\label{eq:m3}
|\partial_1^{\beta_1}\partial_2^{\beta_2}\partial_3^{\beta_3}\sigma_3(\xi_1,\xi_2,\xi_3)|\leq
C m^2(\lambda)\eta^{-4}\lambda^{-\beta_1}\eta^{-\beta_2-\beta_3}.
\end{equation}
\end{proposition}
\begin{proof}
We may assume $\max(|\xi_1|,|\xi_2|,|\xi_3|)\gg 1$, otherwise
$\sigma_3\equiv 0$. Since on the hyperplane $\xi_1+\xi_2+\xi_3=0$,
\[v_3=i(\xi_1^5+\xi_2^5+\xi_3^5)=\frac{5i}{2}\xi_1\xi_2\xi_3(\xi_1^2+\xi_2^2+\xi_3^2)\]
is with a size about $\lambda \eta^4$ and
\[M_3(\xi_1,\xi_2,\xi_3)=-i[m(\xi_1)m(\xi_2+\xi_3)(\xi_2+\xi_3)]_{sym}=i(m^2(\xi_1)\xi_1+m^2(\xi_2)\xi_2+m^2(\xi_3)\xi_3),\]
if $\lambda \sim \eta$, we extend $\sigma_3$ by setting
\begin{equation}
\sigma_3(\xi_1,\xi_2,\xi_3)=-\frac{i(m^2(\xi_1)\xi_1+m^2(\xi_2)\xi_2+m^2(\xi_3)\xi_3)}{\frac{5i}{2}\xi_1\xi_2\xi_3(\xi_1^2+\xi_2^2+\xi_3^2-\frac{6}{5}\mu)},
\end{equation}
and if $\lambda\ll \eta$, we extend $\sigma_3$ by setting
\begin{equation}
\sigma_3(\xi_1,\xi_2,\xi_3)=-\frac{i(m^2(\xi_1)\xi_1+m^2(\xi_2)\xi_2-m^2(\xi_1+\xi_2)(\xi_1+\xi_2))}{\frac{5i}{2}\xi_1\xi_2\xi_3(\xi_1^2+\xi_2^2+\xi_3^2-\frac{6}{5}\mu)}.
\end{equation}
 From \eqref{eq:mvt} and \eqref{eq:eImul}, we see that \eqref{eq:m3} holds.
\end{proof}

Now we give the pointwise bounds for $\sigma_4$ which is key to
estimate the growth of $E^4_I(t)$.
\begin{proposition}
Assume $m$ is of the form \eqref{eq:m}. In the region where
$|\xi_i|\sim N_i,|\xi_j+\xi_k|\sim N_{jk}$ for $N_i, N_{jk}$ dyadic
and  $N_1\geq N_2\geq N_3\geq N_4$,
\begin{equation}\label{eq:m4}
\frac{|M_4(\xi_1,\xi_2,\xi_3,\xi_4)|}{|h_4-v_4|}\les
\frac{m^2(\min(N_i,N_{jk}))}{(N+N_1)^2(N+N_2)^2(N+N_3)^3(N+N_4)}.
\end{equation}
\end{proposition}

\begin{proof}
We will use the ideas in \cite{GuoWang}. From Lemma
\ref{lem:decreso} it suffices to prove
\begin{align*}
\frac{|M_4(\xi_1,\xi_2,\xi_3,\xi_4)|}{|v_4|}\les
\frac{m^2(\min(N_i,N_{jk}))}{(N+N_1)^2(N+N_2)^2(N+N_3)^3(N+N_4)}.
\end{align*}
Since $\xi_1+\xi_2+\xi_3+\xi_4=0$, then $N_1\sim N_2$. We can also
assume that $N_1\sim N_2 \ges N$, otherwise $M_4$ vanishes, since
$m^2(\xi)=1$ if $|\xi|\leq N$. If $\max(N_{12},N_{13},N_{14})\ll
N_1$, then $\xi_3\approx-\xi_1,\ \xi_4\approx -\xi_1$, which
contradicts that $\xi_1+\xi_2+\xi_3+\xi_4=0$. Hence we get
$\max(N_{12},N_{13},N_{14})\sim N_1$. We rewrite the right-hand side
of \eqref{eq:m4} as
\begin{equation}
\frac{m^2(\min(N_i,N_{jk}))}{{N_1}^4(N+N_3)^3(N+N_4)}.
\end{equation}

From Lemma \ref{lem:decreso} we get if $\xi_1+\xi_2+\xi_3+\xi_4=0$
then
\[v_4=\xi_1^5+\xi_2^5+\xi_3^5+\xi_4^5=-\frac{5}{2}(\xi_1+\xi_2)(\xi_1+\xi_3)(\xi_2+\xi_3)({\xi_1^2+\xi_2^2+\xi_3^2}+\xi_4^2)\]
is with size $N_{12}N_{13}N_{14}N_1^2$. From the construction of
$M_4$ we get
\begin{align}\label{eq:rm4}
CM_4(\xi_1,\xi_2,\xi_3,\xi_4)=&[\sigma_3(\xi_1,\xi_2,\xi_3+\xi_4)(\xi_3+\xi_4)]_{sym}\nonumber\\
=&\sigma_3(\xi_1,\xi_2,\xi_3+\xi_4)(\xi_3+\xi_4)+\sigma_3(\xi_1,\xi_3,\xi_2+\xi_4)(\xi_2+\xi_4)\nonumber\\
&+\sigma_3(\xi_1,\xi_4,\xi_2+\xi_3)(\xi_2+\xi_3)+\sigma_3(\xi_2,\xi_3,\xi_1+\xi_4)(\xi_1+\xi_4)\nonumber\\
&+\sigma_3(\xi_2,\xi_4,\xi_1+\xi_3)(\xi_1+\xi_3)+\sigma_3(\xi_3,\xi_4,\xi_1+\xi_2)(\xi_1+\xi_2)\nonumber\\
=&[\sigma_3(\xi_1,\xi_2,\xi_3+\xi_4)-\sigma_3(-\xi_3,-\xi_4,\xi_3+\xi_4)](\xi_3+\xi_4)\nonumber\\
&+[\sigma_3(\xi_1,\xi_3,\xi_2+\xi_4)-\sigma_3(-\xi_2,-\xi_4,\xi_2+\xi_4)](\xi_2+\xi_4)\nonumber\\
&+[\sigma_3(\xi_1,\xi_4,\xi_2+\xi_3)-\sigma_3(-\xi_2,-\xi_3,\xi_2+\xi_3)](\xi_2+\xi_3)\nonumber\\
:=&I+II+III.
\end{align}
The bound \eqref{eq:m4} will follow from case by case analysis.

{\bf Case 1.} $|N_4|\ges \half{N}$.

{\bf Case 1a.} $N_{12}, N_{13}, N_{14}\ges N_1$.

For this case, we just use \eqref{eq:m3}, then we get
\begin{align*}
\frac{|M_4(\xi_1,\xi_2,\xi_3,\xi_4)|}{|v_4|}\les
\frac{m^2(N_4)}{N_1^8},
\end{align*}
which suffices to give the bound \eqref{eq:m4}.

{\bf Case 1b.} $N_{12}\ll N_1$, $N_{13}\ges N_1$, $N_{14}\ges N_1$.

For the contribution of I, we just use \eqref{eq:m3}, then we get
\begin{align*}
\frac{|I|}{|v_4|}\les \frac{m^2(\min(N_4, N_{12}))}{N_1^8},
\end{align*}
which suffices to give the bound \eqref{eq:m4}.

Contribution of II. We first write
\begin{align*}
II=&[\sigma_3(\xi_1,\xi_3,\xi_2+\xi_4)-\sigma_3(-\xi_2,-\xi_4,\xi_2+\xi_4)](\xi_2+\xi_4)\nonumber\\
=&[\sigma_3(\xi_1,\xi_3,\xi_2+\xi_4)-\sigma_3(-\xi_2,\xi_3,\xi_2+\xi_4)](\xi_2+\xi_4)\nonumber\\
&+[\sigma_3(-\xi_2,\xi_3,\xi_2+\xi_4)-\sigma_3(-\xi_2,-\xi_4,\xi_2+\xi_4)](\xi_2+\xi_4)\nonumber\\
=&II_1+II_2.
\end{align*}
If $N_{12}\ges N_3$, then using \eqref{eq:mvt}, \eqref{eq:m3} for
$II_1$ and using \eqref{eq:m3} for $II_2$, we get
\[\frac{|II|}{|v_4|}\les  \frac{m^2(N_4)}{N_1^7 N_3}.\]
If $N_{12}\ll N_3$, using \eqref{eq:mvt}, \eqref{eq:m3} for both
$II_1$ and $II_2$, then we get
\[\frac{|II|}{|v_4|}\les  \frac{m^2(N_4)}{N_1^7 N_3}, \]

Contribution of III. This is identical to II.

{\bf Case 1c.} $N_{12}\ll N_1$, $N_{13}\ll N_1$, $N_{14}\ges N_1$.

Since $N_{12}\ll N_1$, $N_{13}\ll N_1$, then $N_1\sim N_2\sim
N_3\sim N_4$.

Contribution of I. We first write
\begin{align*}
I=&[\sigma_3(\xi_1,\xi_2,\xi_3+\xi_4)-\sigma_3(-\xi_3,\xi_2,\xi_3+\xi_4)](\xi_3+\xi_4)\nonumber\\
&+[\sigma_3(-\xi_3,\xi_2,\xi_3+\xi_4)-\sigma_3(-\xi_3,-\xi_4,\xi_3+\xi_4)](\xi_3+\xi_4)\nonumber\\
:=&I_1+I_2.
\end{align*}
Using \eqref{eq:m3}, \eqref{eq:mvt} for both $I_1$ and $I_2$, then
we get
\begin{align*}
\frac{|I|}{|v_4|}\les\frac{m^2(N_{12})}{N_1^8}.
\end{align*}

Contribution of II. This is identical to I.

Contribution of III. We first write
\begin{align*}
III=&[\sigma_3(\xi_1,\xi_4,\xi_2+\xi_3)-\sigma_3(-\xi_2,-\xi_3,\xi_2+\xi_3)](\xi_2+\xi_3)\nonumber\\
=&\frac{1}{2}[\sigma_3(\xi_1,\xi_4,\xi_2+\xi_3)-\sigma_3(-\xi_2,-\xi_3,\xi_2+\xi_3)\nonumber\\
&-\sigma_3(-\xi_3,-\xi_2,\xi_2+\xi_3)+\sigma_3(\xi_4,\xi_1,\xi_2+\xi_3)](\xi_2+\xi_3).
\end{align*}
Using \eqref{eq:dmvt} five times, we have
\begin{align*}
\frac{|III|}{|v_4|}\les\frac{m^2(N_{1})}{N_1^8}.
\end{align*}

{\bf Case 1d.} $N_{12}\ll N_1$, $N_{13}\ges N_1$, $N_{14}\ll N_1$.

This case is identical to Case 1c.

{\bf Case 2.} $N_4\ll N/2$.

In this case we have $m^2(\min(N_i,N_{jk}))=1$, and $N_{13}\sim
|\xi_1+\xi_3|=|\xi_2+\xi_4|\sim N_1$. We discuss this case in the
following two subcases.

{\bf Case 2a.} $N_1/4>N_{12}\ges N/2$.

Since $N_4\ll N/2$ and $|\xi_3+\xi_4|=|\xi_1+\xi_2|\ges N/2$, then
$N_3\ges N/2$. From $|v_4|\sim N_{12}N_1^4$, then we bound the six
terms in \eqref{eq:rm4} respectively, and get
\begin{align*}
\frac{|M_4|}{|v_4|}\les \frac{1}{N_1^4N_3^3N},
\end{align*}
which suffices to give the bound \eqref{eq:m4}.

{\bf Case 2b.} $N_{12}\ll N/2$.

Since $N_{12}=N_{34}\ll N/2$ and $N_4\ll N/2$, then we must have
$N_3\ll N/2$, and $N_{13}\sim N_{14}\sim N_1$.

Contribution of I. Since $N_3, N_4, N_{34}\ll N/2$, then we have
$\sigma_3(-\xi_3,-\xi_4,\xi_3+\xi_4)=0$. Thus it follows from
\eqref{eq:m3} that
\begin{align*}
\frac{|I|}{|v_4|}\les
\frac{|\sigma_3(\xi_1,\xi_2,\xi_3+\xi_4)|}{N_1^4}\les
\frac{1}{N_1^8}.
\end{align*}

Contribution of II and III. We have two items of $N_3, N_4, N_{12}$
in the denominator, which will cause a problem. Thus we can't deal
with II and III separately, but we need to exploit the cancelation
between II and III. We rewrite
\begin{align*}
II+III=&[\sigma_3(\xi_1,\xi_3,\xi_2+\xi_4)-\sigma_3(-\xi_2,-\xi_4,\xi_2+\xi_4)](\xi_2+\xi_4)\nonumber\\
&+[\sigma_3(\xi_1,\xi_4,\xi_2+\xi_3)-\sigma_3(-\xi_2,-\xi_3,\xi_2+\xi_3)](\xi_2+\xi_3)\nonumber\\
=&[\sigma_3(\xi_1,\xi_3,\xi_2+\xi_4)-\sigma_3(-\xi_2,-\xi_4,\xi_2+\xi_4)]\xi_4\nonumber\\
&+[\sigma_3(\xi_1,\xi_4,\xi_2+\xi_3)-\sigma_3(-\xi_2,-\xi_3,\xi_2+\xi_3)]\xi_3\nonumber\\
&+[\sigma_3(\xi_1,\xi_3,\xi_2+\xi_4)-\sigma_3(-\xi_2,-\xi_4,\xi_2+\xi_4)\nonumber\\
&\quad+\sigma_3(\xi_1,\xi_4,\xi_2+\xi_3)-\sigma_3(-\xi_2,-\xi_3,\xi_2+\xi_3)]\xi_2\nonumber\\
=&J_1+J_2+J_3.
\end{align*}
We first consider $J_1$. From
\begin{align*}
\frac{|J_1|}{|v_4|}\les
\frac{|[\sigma_3(\xi_1,\xi_3,\xi_2+\xi_4)-\sigma_3(-\xi_2,-\xi_4,\xi_2+\xi_4)]\xi_4|}{N_{12}N_1^4}
\end{align*}
if $N_{12}\ll N_3$ (in this case, $N_3\sim N_4$), using
\eqref{eq:mvt} twice, otherwise using \eqref{eq:mvt} once and
\eqref{eq:m3}, then we get
\begin{align*}
\frac{|J_1|}{|v_4|}\les \frac{1}{N_1^8}.
\end{align*}
The term $J_2$ is identical to the term $J_1$. Now we consider
$J_3$. We first assume that $N_{12}\ges N_3$. Then by the symmetry
of $\sigma_3$, we get
\begin{align*}
J_3=&[\sigma_3(\xi_1,\xi_3,\xi_2+\xi_4)-\sigma_3(-\xi_2-\xi_3,\xi_3,\xi_2)\nonumber\\
&+\sigma_3(\xi_1,\xi_4,\xi_2+\xi_3)-\sigma_3(-\xi_2-\xi_4,\xi_4,\xi_2)]\xi_2.
\end{align*}
From \eqref{eq:mvt} and $N_{12}\ges N_3$, we get
\begin{align*}
\frac{|J_3|}{|v_4|}\les \frac{1}{N_1^8}.
\end{align*}
If $N_{12}\ll N_3$, then $N_3\sim N_4$. We first write
\begin{align*}
J_3=&[\sigma_3(-\xi_2,\xi_3,\xi_2+\xi_4)-\sigma_3(-\xi_2,-\xi_4,\xi_2+\xi_4)\nonumber\\
&\quad+\sigma_3(\xi_1,\xi_4,\xi_2+\xi_3)-\sigma_3(\xi_1,-\xi_3,\xi_2+\xi_3)]\xi_2\nonumber\\
&+[\sigma_3(\xi_1,\xi_3,\xi_2+\xi_4)-\sigma_3(-\xi_2,\xi_3,\xi_2+\xi_4)\nonumber\\
&\quad+\sigma_3(\xi_1,-\xi_3,\xi_2+\xi_3)-\sigma_3(-\xi_2,-\xi_3,\xi_2+\xi_3)]\xi_2\nonumber\\
=&J_{31}+J_{32}.
\end{align*}
It follows from \eqref{eq:mvt} that
\begin{align*}
\frac{|J_{32}|}{|v_4|}\les \frac{1}{N_1^8}.
\end{align*}
It remains to bound $J_{31}$. It follows from \eqref{eq:m3} and
$m^2(\xi_3)=m^2(\xi_4)=1$ that
\begin{align*}
CJ_{31}=&C[{\sigma}_3(-\xi_2,\xi_3,\xi_2+\xi_4)-{\sigma}_3(-\xi_2,-\xi_4,\xi_2+\xi_4)\nonumber\\
&\quad-{\sigma}_3(\xi_1,-\xi_3,\xi_2+\xi_3)+{\sigma}_3(\xi_1,\xi_4,\xi_2+\xi_3)]\xi_2\nonumber\\
=&\frac{-m^2(\xi_2)\xi_2+\xi_3+m^2(\xi_2+\xi_4)(\xi_2+\xi_4)}{-\xi_2\xi_3(\xi_2+\xi_4)[\xi_2^2+\xi_3^2+(\xi_2+\xi_4)^2]}\xi_2\nonumber\\
&\quad
-\frac{-m^2(\xi_2)\xi_2-\xi_4+m^2(\xi_2+\xi_4)(\xi_2+\xi_4)}{\xi_2\xi_4(\xi_2+\xi_4)[\xi_2^2+\xi_4^2+(\xi_2+\xi_4)^2]}\xi_2\nonumber\\
&+\frac{m^2(\xi_1)\xi_1+\xi_4+m^2(\xi_2+\xi_3)(\xi_2+\xi_3)}{\xi_1\xi_4(\xi_2+\xi_3)[\xi_1^2+\xi_4^2+(\xi_2+\xi_3)^2]}\xi_2\nonumber\\
&\quad
-\frac{m^2(\xi_1)\xi_1-\xi_3+m^2(\xi_2+\xi_3)(\xi_2+\xi_3)}{-\xi_1\xi_3(\xi_2+\xi_3)[\xi_1^2+\xi_3^2+(\xi_2+\xi_3)^2]}\xi_2.
\end{align*}
We rewrite $J_{31}$ as following
\begin{align*}
&\big\{\frac{-m^2(\xi_2)\xi_2+m^2(\xi_2+\xi_4)(\xi_2+\xi_4)}{-\xi_2\xi_3(\xi_2+\xi_4)[\xi_2^2+\xi_3^2+(\xi_2+\xi_4)^2]}\xi_2-\frac{-m^2(\xi_2)\xi_2+m^2(\xi_2+\xi_4)(\xi_2+\xi_4)}{\xi_2\xi_4(\xi_2+\xi_4)[\xi_2^2+\xi_4^2+(\xi_2+\xi_4)^2]}\xi_2\\
&-\frac{m^2(\xi_1)\xi_1+m^2(\xi_2+\xi_3)(\xi_2+\xi_3)}{-\xi_1\xi_3(\xi_2+\xi_3)[\xi_1^2+\xi_3^2+(\xi_2+\xi_3)^2]}\xi_2+\frac{m^2(\xi_1)\xi_1+m^2(\xi_2+\xi_3)(\xi_2+\xi_3)}{\xi_1\xi_4(\xi_2+\xi_3)[\xi_1^2+\xi_4^2+(\xi_2+\xi_3)^2]}\xi_2\big\}\\
&+\{\frac{\xi_2}{\xi_1(\xi_2+\xi_3)[\xi_1^2+\xi_4^2+(\xi_2+\xi_3)^2]}-\frac{\xi_2}{\xi_1(\xi_2+\xi_3)[\xi_1^2+\xi_3^2+(\xi_2+\xi_3)^2]}\\
&-\frac{1}{(\xi_2+\xi_4)[\xi_2^2+\xi_3^2+(\xi_2+\xi_4)^2]}+\frac{1}{(\xi_2+\xi_4)[\xi_2^2+\xi_4^2+(\xi_2+\xi_4)^2]}\}\\
&:=J_{311}+J_{312}.
\end{align*}
We consider first the term $J_{312}$.
\begin{align*}
J_{312}=&\{\frac{\xi_2(\xi_3-\xi_4)(\xi_3+\xi_4)}{\xi_1(\xi_2+\xi_3)[\xi_1^2+\xi_4^2+(\xi_2+\xi_3)^2][\xi_1^2+\xi_3^2+(\xi_2+\xi_3)^2]}\\
&+\frac{(\xi_3-\xi_4)(\xi_3+\xi_4)}{(\xi_2+\xi_4)[\xi_2^2+\xi_3^2+(\xi_2+\xi_4)^2][\xi_2^2+\xi_4^2+(\xi_2+\xi_4)^2]}\}
\end{align*}
Thus we get
\[\frac{|J_{312}|}{|v_4|}\les \frac{1}{N_1^8}.\]
It remains to bound the term $J_{311}$. We will compare it with the
following term denoted by $J_{311}'$:
\begin{align*}
&\{\frac{-m^2(\xi_2)\xi_2+m^2(\xi_2+\xi_4)(\xi_2+\xi_4)}{-\xi_2\xi_3(\xi_2+\xi_4)[\xi_2^2+\xi_4^2+(\xi_2+\xi_4)^2]}\xi_2-\frac{-m^2(\xi_2)\xi_2+m^2(\xi_2+\xi_4)(\xi_2+\xi_4)}{\xi_2\xi_4(\xi_2+\xi_4)[\xi_2^2+\xi_4^2+(\xi_2+\xi_4)^2]}\xi_2\\
&+\frac{m^2(\xi_1)\xi_1+m^2(\xi_2+\xi_3)(\xi_2+\xi_3)}{\xi_1\xi_3(\xi_2+\xi_3)[\xi_1^2+\xi_4^2+(\xi_2+\xi_3)^2]}\xi_2+\frac{m^2(\xi_1)\xi_1+m^2(\xi_2+\xi_3)(\xi_2+\xi_3)}{\xi_1\xi_4(\xi_2+\xi_3)[\xi_1^2+\xi_4^2+(\xi_2+\xi_3)^2]}\xi_2\}.
\end{align*}
It is easy to see that as for the term $J_{312}$ we have
\[\frac{|J_{311}-J_{311}'|}{|v_4|}\les \frac{1}{N_1^8}.\]
Thus it remains to show that
\[\frac{|J_{311}'|}{|v_4|}\les \frac{1}{N_1^8}.\]
We rewrite $J_{311}'$ as following
\begin{align*}
-&\frac{\xi_3+\xi_4}{\xi_3\xi_4}\frac{-m^2(\xi_2)\xi_2+m^2(\xi_2+\xi_4)(\xi_2+\xi_4)+m^2(\xi_1)\xi_1+m^2(\xi_2+\xi_3)(\xi_2+\xi_3)}{\xi_2(\xi_2+\xi_4)[\xi_2^2+\xi_4^2+(\xi_2+\xi_4)^2]}\xi_2\\
+&\frac{\xi_3+\xi_4}{\xi_3\xi_4}[m^2(\xi_1)\xi_1+m^2(\xi_2+\xi_3)(\xi_2+\xi_3)]\\
&\times
[\frac{1}{\xi_1(\xi_2+\xi_3)[\xi_1^2+\xi_4^2+(\xi_2+\xi_3)^2]}+\frac{1}{\xi_2(\xi_2+\xi_4)[\xi_2^2+\xi_4^2+(\xi_2+\xi_4)^2]}]\xi_2.
\end{align*}
Therefore, we use \eqref{eq:dmvt} for the first term, and
\eqref{eq:mvt} for the second term, and finally we conclude that
\begin{align*}
\frac{|J_{311}'|}{|v_4|}\les \frac{1}{N_1^8},
\end{align*}
which completes the proof of the proposition.
\end{proof}

From the estimates of $\sigma_4$ we can immediately get the
following
\begin{proposition}\label{prop:lowchM5}
Assume $m$ is of the form \eqref{eq:m}, then
\begin{align*}
|M_5(\xi_1,\ldots,\xi_5)|\les
\left[\frac{m^2(N_{*45})N_{45}}{(N+N_1)^2(N+N_2)^2(N+N_3)^3(N+N_{45})}\right]_{sym},
\end{align*}
where
\[N_{*45}=\min(N_1,N_2,N_3,N_{45},N_{12},N_{13},N_{23}).\]
\end{proposition}

\section{G.W.P. of fifth order KdV on $\R$}

In this section we extend the local solution to a global one. We
will rely on a variant well-posedness result which can be proved
similarly as for the Theorem \ref{thm:lwp}.
\begin{proposition}\label{prop:lowchvlwp}
Let $-7/4\leq s\leq 0$. Assume $\phi$ satisfy
$\norm{I\phi}_{L^2(\R)}\leq 2\epsilon_0\ll 1$. Then equation
\eqref{scaledkawa} has a unique solution on $[-1,1]$ such that
\begin{eqnarray}
\norm{Iu}_{\bar{F}^0(1)}\leq C\epsilon_0,
\end{eqnarray}
where $C$ is independent of $N$ and $0<\lambda\leq 1$.
\end{proposition}

From Proposition \ref{prop:lowchvlwp}, we see it suffices to control
the growth of $E_I^2(t)$. It is better controlling directly the
growth of $E_I^4(t)$, and then using the following proposition we
can that of $E_I^2(t)$.
\begin{proposition}\label{prop:EI2EI4}
Let $I$ be defined with the multiplier $m$ of the form \eqref{eq:m}
and $s=-7/4$. Then
\begin{equation}
|E_I^4(t)-E_I^2(t)|\les \norm{Iu(t)}_{L^2}^3+\norm{Iu(t)}_{L^2}^4.
\end{equation}
\end{proposition}

\begin{proof}
Since $E_I^4(t)=E_I^2(t)+\Lambda_3(\sigma_3)+\Lambda_4(\sigma_4)$,
then it suffices to show
\begin{eqnarray}
|\Lambda_3(\sigma_3;u_1,u_2,u_3)|&\les&
\prod_{i=1}^3\norm{Iu_i}_{L^2},\label{eq:endiff3linear}\\
|\Lambda_4(\sigma_4;u_1,u_2,u_3,u_4)|&\les&
\prod_{i=1}^4\norm{Iu_i}_{L^2}.\label{eq:endiff4linear}
\end{eqnarray}
We may assume that $\wh{u_i}$ are non-negative. To prove
\eqref{eq:endiff3linear}, it suffices to prove
\begin{align}\label{eq:lowchE3linearprf1}
\left|{\Lambda_3\brk{\frac{m^2(\xi_1)\xi_1+m^2(\xi_2)\xi_2+m^2(\xi_3)\xi_3}{\xi_1\xi_2\xi_3(\xi_1^2+\xi_2^2+\xi_3^2)m(\xi_1)m(\xi_2)m(\xi_3)};u_1,u_2,u_3}}\right|\les
\prod_{i=1}^3\norm{u_i}_2.
\end{align}
By Littlewood-Paley decomposition, we get the left-hand side of
\eqref{eq:lowchE3linearprf1} is bounded by
\begin{align}\label{eq:lowchE3linearprf2a}
\sum_{k_i\geq
0}\left|{\Lambda_3\brk{\frac{m^2(\xi_1)\xi_1+m^2(\xi_2)\xi_2+m^2(\xi_3)\xi_3}{\xi_1\xi_2\xi_3(\xi_1^2+\xi_2^2+\xi_3^2)m(\xi_1)m(\xi_2)m(\xi_3)};P_{k_1}u_1,P_{k_2}u_2,P_{k_3}u_3}}\right|.
\end{align}
Let $N_i=2^{k_i}$. From symmetry we may assume $N_1\geq N_2 \geq
N_3$, and hence $N_1\sim N_2\ges N$.

Case 1. $N_3\ll N$.

In this case $m(N_3)=1$, then we get
\begin{align*}
\eqref{eq:lowchE3linearprf2a}\les& \sum_{k_i\geq
0}\left|{\Lambda_3\brk{\frac{N^sN^s}{N_1^{1+s}N_1^{1+s}N_1^{2}};P_{k_1}u_1,P_{k_2}u_2,P_{k_3}u_3}}\right|\\
\les& \sum_{k_i\geq
0}\left|{\Lambda_3\brk{N_1^{-1/4}N_2^{-1/4};P_{k_1}u_1,P_{k_2}u_2,P_{k_3}u_3}}\right|.
\end{align*}
It suffices to prove
\begin{align*}
\sum_{k_i\geq 0}\int_{\xi_1+\xi_2+\xi_3=0,|\xi_i|\sim
N_i}N_1^{-1/2}\prod_{i=1}^3\eta_{k_i}(\xi_i)\wh{u_i}(\xi_i)\les
\prod_{i=1}^3\norm{u_i}_{L^2}.
\end{align*}
Define $v_i(x)$ as following:
\[\wh{v_i}(\xi)=N_i^{-1/6}\wh{u_i}(\xi)\chi_{\{|\xi|\sim N_i\}}(\xi).\]
By Sobolev embedding inequality we have $\norm{v_i}_{L^3}\les
\norm{u_i}_{L^2}$, thus using H\"older's inequality we get
\begin{align*}
&\sum_{k_i\geq 0}\int_{\xi_1+\xi_2+\xi_3=0,|\xi_i|\sim
N_i}N_1^{-1/2}\prod_{i=1}^3\eta_{k_i}(\xi_i)\wh{u_i}(\xi_i)\\
&\les\sum_{k_i\geq
0}N_1^{-1/6}N_3^{1/6}\prod_{i=1}^3\norm{v_i}_{L^3}\les\prod_{i=1}^3\norm{u_i}_{L^2}.
\end{align*}

Case 2. $N_3\ges N$. It is obvious that
\begin{align*}
\eqref{eq:lowchE3linearprf2a}\les \sum_{k_i\geq
0}\left|{\Lambda_3\brk{\frac{N_3^{-7/4}N^{-7/4}}{N_1^{1/2}};P_{k_1}u_1,P_{k_2}u_2,P_{k_3}u_3}}\right|\les
\prod_{i=1}^3\norm{u_i}_{L^2}.
\end{align*}
Thus we get \eqref{eq:endiff3linear}.

Next we show \eqref{eq:endiff4linear}. It suffices to prove
\begin{eqnarray}\label{eq:lowchE4linearprf1}
\left|{\Lambda_4\brk{\frac{\sigma_4}{m(\xi_1)m(\xi_2)m(\xi_3)m(\xi_4)};u_1,u_2,u_3,u_4}}\right|\les
\prod_{i=1}^4\norm{u_i}_2.
\end{eqnarray}
By Littlewood-Paley decomposition we get the left-hand side of
\eqref{eq:lowchE4linearprf1} is dominated by
\begin{eqnarray}\label{eq:lowchE3linearprf2}
\sum_{k_i\geq
0}\left|{\Lambda_4\brk{\frac{\sigma_4}{m(\xi_1)m(\xi_2)m(\xi_3)m(\xi_4)};P_{k_1}u_1,P_{k_2}u_2,P_{k_3}u_3,P_{k_4}u_4}}\right|.
\end{eqnarray}
Let $N_i=2^{k_i}$. From symmetry we may assume $N_1\geq N_2 \geq
N_3\geq N_4$, hence we may assume $N_1\sim N_2\ges N$. Since
\[\left|\frac{\sigma_4}{m(\xi_1)m(\xi_2)m(\xi_3)m(\xi_4)}\right|\les \frac{1}{\prod_{i=1}^4(N+N_i)^2m(N_i)}\les \frac{N^{-7}}{\prod_{i=1}^4N_i^{1/4}}.\]
using H\"older's inequality we get
\begin{align*}
\eqref{eq:lowchE3linearprf2}\les& \sum_{k_i\geq
0}\frac{N^{-7}}{\prod_{i=1}^4N_i^{1/4}}\norm{P_{k_1}u_1}_{L^2}\norm{P_{k_2}u_2}_{L^2}\norm{P_{k_3}u_3}_{L^\infty}\norm{P_{k_4}u_4}_{L^\infty}\\
\les&\prod_{i=1}^4\norm{u_i}_2.
\end{align*}
Therefore, we complete the proof of the proposition.
\end{proof}

Since $E_I^2(t)$ is very close to $E_I^4(t)$, then we will control
$E_I^4(t)$ and hence control $E_I^2(t)$. In order to control the
increase of $E_I^4(t)$, we need to control its derivative
\[\frac{d}{dt}E_I^4(t)=\Lambda_5(M_5),\]
where
\[M_5(\xi_1,\ldots,\xi_5)=-2i[\sigma_4(\xi_1,\xi_2,\xi_3,\xi_4+\xi_5)(\xi_4+\xi_5)]_{sym}.\]

\begin{proposition}\label{prop:5linear}
Assume $I\subset \R$ with $|I|\les 1$. Let $0\leq k_1\leq k_2\leq
k_3\leq k_4 \leq k_5$ and $k_4\geq 10$. Then we have
\begin{align}\label{eq:5linear1}
&\left|\int_I\int
P_{k_{12}}\big(P_{k_1}(w_1)P_{k_2}(w_2)\big)\prod_{i=3}^5
P_{k_i}(w_i)(x,t)dxdt\right|\nonumber\\
&\les
2^{\frac{5}{4}k_1}2^{k_2/4}2^{k_3/4}2^{-2k_4}2^{-2k_5}\prod_{j=1}^5\norm{\wh{P_{k_j}(w_j)}}_{X_{k_j}},
\end{align}
where if $k_j=0$ then $X_{k_j}$ is replaced by $\bar{X}_{k_j}$ on
the right-hand side.
\end{proposition}
\begin{proof}
From H\"older's inequality the left-hand side of \eqref{eq:5linear1}
is dominated by
\[\norm{P_{k_1}(w_1)}_{L_x^2L_{t\in I}^\infty}\norm{P_{k_2}(w_2)}_{L_x^4L_{t\in I}^\infty}\norm{P_{k_2}(w_2)}_{L_x^4L_{t\in I}^\infty}
\norm{P_{k_4}(w_4)}_{L_x^\infty
L_{t}^2}\norm{P_{k_5}(w_5)}_{L_x^\infty L_{t}^2}.\] Then the
proposition follows immediately from Proposition \ref{Xkembedding}.
\end{proof}

\begin{proposition}\label{prop:inEI4}
Let $\delta\les 1$. Assume $m$ is of the form \eqref{eq:m} with
$s=-7/4$, then
\begin{eqnarray}
\left|\int_0^\delta \Lambda_5(M_5)dt\right|\les
N^{-\frac{35}{4}}\norm{Iu}^5_{\bar F^{0}(\delta)}.
\end{eqnarray}
\end{proposition}
\begin{proof}
By the definitions, it suffices to prove that
\begin{align*}
\left|\int_0^\delta
\Lambda_5\brk{\frac{\sigma_4(\xi_1,\xi_2,\xi_3,\xi_4+\xi_5)(\xi_4+\xi_5)}{m(\xi_1)m(\xi_2)m(\xi_3)m(\xi_4)m(\xi_5)}}dt\right|\les
N^{-\frac{35}{4}}\norm{u}^5_{\bar F^{0}(\delta)}.
\end{align*}
By the Littlewood-Paley decomposition $u=\sum_{k\geq 0}P_ku$, it
suffices to prove
\begin{eqnarray*}
&&\sum_{N_1,\cdots, N_5,N_{45}\geq 0}\left|\int_0^\delta
\Lambda_5\left(\frac{\sigma_4(\xi_1,\xi_2,\xi_3,\xi_4+\xi_5)(\xi_4+\xi_5)}{m(\xi_1)m(\xi_2)m(\xi_3)m(\xi_4)m(\xi_5)};P_{k_1}u,\cdots,P_{k_5}u\right)dt\right|\\
&&\les N^{-\frac{35}{4}}\norm{u}^5_{\bar{F}^0(\delta)},
\end{eqnarray*}
where $N_i=2^{k_i}$ and $|\xi_4+\xi_5|\sim N_{45}$ for $N_{45}$
dyadic. From symmetry we may assume $N_1\geq N_2\geq N_3$ and
$N_4\geq N_5$ and two of the $N_i\ges N$. We fix the extension
$\wt{u}_i$ such that $\norm{\wt{u}_i}_{\bar{F}^0}\les
2\norm{{u}_i}_{\bar{F}^0(\delta)}$. For simplicity, we still denote
$u_i$.

{\bf Case 1.} $N_1\sim N_2\ges N$ and $N_4\les N_2$.

{\bf Case 1a.} $N_{45}\ges N_3$.

 From the form
\eqref{eq:m} with $s=-7/4$ we get that
$\frac{1}{(N+N_i)^2m(N_i)}\les N^{-7/4}{N_i}^{-1/4}$ and
$\frac{1}{m(N_3)m(N_4)m(N_5)}\les
N^{-21/4}N_3^{7/4}N_4^{7/4}N_5^{7/4}$. Thus we have
\[\left|\frac{\sigma_4(\xi_1,\xi_2,\xi_3,\xi_4+\xi_5)(\xi_4+\xi_5)}
{m(\xi_1)m(\xi_2)m(\xi_3)m(\xi_4)m(\xi_5)}\right|\les
\frac{N^{-35/4}N_1^{-1/2}N_3^{7/4}N_4^{7/4}N_5^{7/4}m^2(\min(N_i,N_{jk}))N_{45}}{(N+N_3)(N+N_{45})^3}\]
Therefore in this case we need to control
\begin{align}\label{eq:deriE4}
N^{-\frac{35}{4}}\sum_{N_i,N_{45}}\int_0^\delta \Lambda_5\left(
\frac{N_1^{-1/2}N_3^{7/4}N_4^{7/4}N_5^{7/4}N_{45}}{(N+N_3)(N+N_{45})^3};P_{k_1}u,\cdots,P_{k_5}u\right)dt.
\end{align}
We consider the worst case $N_1\geq N_2\geq N_4\geq N_5\geq N_3$.
From \eqref{eq:5linear1} we get
\begin{align*}
\eqref{eq:deriE4}\les&N^{-\frac{35}{4}}\sum_{N_i}
\frac{N_1^{-1/2}N_3^{7/4}N_4^{7/4}N_5^{7/4}N_{45}}{(N+N_3)(N+N_{45})^3}N_1^{-4}N_3^{5/4}N_4^{1/4}N_5^{1/4}\prod_{i=1}^5
\norm{\wh{P_{k_i}u}}_{X_{k_i}}\\
\les& N^{-\frac{35}{4}}\norm{u}^5_{\bar F^{0}(\delta)}.
\end{align*}

{\bf Case 1b.} $N_{45}\ll N_3$.

We have
\[\left|\frac{\sigma_4(\xi_1,\xi_2,\xi_3,\xi_4+\xi_5)(\xi_4+\xi_5)}
{m(\xi_1)m(\xi_2)m(\xi_3)m(\xi_4)m(\xi_5)}\right|\les
\frac{N^{-35/4}N_1^{-1/2}N_3^{7/4}N_4^{7/4}N_5^{7/4}m^2(\min(N_i,N_{jk}))N_{45}}{(N+N_{45})(N+N_{3})^3}\]
Therefore in this case we need to control
\begin{align}\label{eq:5linear1b}
N^{-\frac{35}{4}}\sum_{N_i,N_{45}}\int_0^\delta \Lambda_5\left(
\frac{N_1^{-1/2}N_3^{7/4}N_4^{7/4}N_5^{7/4}}{(N+N_{3})^3};P_{k_1}u,\cdots,P_{k_5}u\right)dt.
\end{align}
We get from Proposition \ref{prop:5linear} that (still consider the
worst case $N_1\geq N_2\geq N_4\geq N_5\geq N_3$)
\[\eqref{eq:5linear1b}\les N^{-\frac{35}{4}}\sum_{N_i}\frac{N_1^{-1/2}N_3^{7/4}N_4^{7/4}N_5^{7/4}}{(N+N_{3})^3}N_1^{-4}N_3^{5/4}N_4^{1/4}N_5^{1/4}\prod_{i=1}^5
\norm{\wh{P_{k_i}u}}_{X_{k_i}}\les N^{-\frac{35}{4}}\norm{u}^5_{\bar
F^{0}(\delta)}.\]

The rest cases $N_4\sim N_5 \ges N$, $N_1\les N_5$ or $N_1\sim N_4
\ges N$ follow in a similar ways. We omit the details.
\end{proof}

For any fixed $u_0 \in H^{-7/4}$ and time $T>0$, our goal is to
construct the solution to \eqref{kawa} on $t\in [0,T]$. If $u$ is a
solution to \eqref{kawa} with initial data $u_0$, then for any
$\lambda>0$, $u_\lambda(x,t)=\lambda^{4}u(\lambda x, \lambda^5 t)$
is a solution to \eqref{scaledkawa} with initial data
$u_{0,\lambda}=\lambda^{4}u_0(\lambda x)$. By simple calculation we
know
\[
\norm{Iu_{0,\lambda}}_{L^2}\les
\lambda^{{7/4}}N^{7/4}\norm{u_0}_{H^s}.
\]
For fixed $N$ ($N$ will be determined later), we take $\lambda\sim
N^{-1}$ such that
\begin{align*}
\lambda^{{7/4}}N^{7/4}\norm{\phi}_{H^{-7/4}}=\epsilon_0<1.
\end{align*}
For the simplicity of notations, we still denote $u_\lambda$ by $u$,
$u_{0,\lambda}$ by $u_0$, and assume $\norm{Iu_0}_{L^2}\leq
\epsilon_0$, then the goal is to construct the solution to
\eqref{scaledkawa} on $[0,\lambda^{-5} T]$. According to Proposition
\ref{prop:lowchvlwp} we get a local solution $u$ on $t\in [0,1]$,
then we need to control the modified energy
$E_I^2(t)=\norm{Iu}^2_{L^2}$.

First we see the control of $E_I^2(t)$ for $t\in [0,1]$, we will
prove that $E_I^2(t)<4\epsilon_0^2$. Using bootstrap we may assume
$E_I^2(t)<5\epsilon_0^2$, then from Proposition \ref{prop:EI2EI4} we
get
\[E_I^4(0)=E_I^2(0)+O(\epsilon_0^3)\]
and
\[E_I^4(t)=E_I^2(t)+O(\epsilon_0^3).\]
Thus from Proposition \ref{prop:inEI4} we get for $t\in [0,1]$
\[E_I^4(t)\leq E_I^4(0)+C\epsilon_0^5N^{-35/4}.\]
Therefore
\begin{align*}
\norm{Iu(1)}_{L^2}^2=E_I^4(1)+O(\epsilon_0^3)\leq&
E_I^4(0)+C\epsilon_0^5N^{-35/4}+O(\epsilon_0^3)\\
=&\epsilon_0^2+C\epsilon_0^5N^{-35/4}+O(\epsilon_0^3)<4\epsilon_0^2.
\end{align*}
Thus $u$ can be extended to $t\in [0,2]$. Extending as this
$M$-steps, we get for $t\in [0,M+1]$
\begin{align*}
E_I^4(t)\leq E_I^4(0)+CM\epsilon_0^5N^{-35/4}.
\end{align*}
Thus as long as $MN^{-35/4}\les 1$, then we have
\[E_I^2(M)=E_I^4(M)+O(\epsilon_0^3)=\epsilon_0^2+O(\epsilon_0^3)+CM\epsilon_0^5N^{-35/4}<4\epsilon_0^2,\]
Therefore, the solution can be extended to $t\in [0,N^{35/4}]$.
Taking $N(T)$ sufficiently large such that
\[N^{35/4}>\lambda^{-5} T\sim N^{5}T.\]
Thus $u$ is extended to $[0,\lambda^{-5} T]$, then for the original
equation \eqref{kawa}, using the scaling, we prove its solution
extends to $[0,T]$.

At the end of this section, we see what we know about the global
solution. Using the scaling we get
\begin{align*}
\sup_{t\in [0,T]}\norm{u(t)}_{H^{-7/4}}\sim&
\lambda^{{-7/4}}\sup_{t\in
[0,\lambda^{-5}T]}\norm{u_\lambda(t)}_{H^{-7/4}}\leq
\lambda^{-{7/4}}\sup_{t\in
[0,\lambda^{-5}T]}\norm{Iu_\lambda(t)}_{L^2},\\
&\norm{I\phi_\lambda}_{L^2}\les
N^{7/4}\norm{\phi_\lambda}_{H^{-7/4}}\sim
N^{7/4}\lambda^{{7/4}}\norm{\phi}_{H^{-7/4}}.
\end{align*}
From the previous proof we know
\[\sup_{t\in [0,\lambda^{-5} T]}\norm{Iu_\lambda(t)}_{L^2}\les \norm{I\phi_\lambda}_{L^2},\]
thus we get
\[\sup_{t\in [0,T]}\norm{u(t)}_{H^{-7/4}}\les N^{{7/4}}\norm{\phi}_{H^{-7/4}}.\]
Take $\lambda$ such that $\norm{I\phi_\lambda}_{L^2}\sim \epsilon_0
\ll 1$, then
\[\lambda=\lambda(N,\epsilon_0,\norm{\phi}_{H^{-7/4}})\sim
\brk{\frac{\norm{\phi}_{H^{-7/4}}}{\epsilon_0}}^{-4/7}N^{-1}.\] We
will choose $N$ such that $N^{\frac{35}{4}} > \lambda^5
T\sim_{c\norm{\phi}_{H^s},\epsilon_0}N^{5}T$, then we get $N\sim
T^{4/15}$. Therefore, we get that the obtained global solution
$u(x,t)$ satisfies
\[\norm{u(t)}_{H^{-7/4}}\les \ (1+|t|)^{7/15}\norm{\phi}_{H^{-7/4}}^{4/3}.\]
We can prove a similar results for $s>-7/4$.

\section{Ill-posedness of the equation}
\setcounter{equation}{0} \label{sec:2}

In this section, we prove an ill-posedness results by following the
method of Bourgain \cite{Bo97}.

\begin{thm}\label{major2}
For $s<-\frac 94$ the solution map of the Cauchy problem
(\ref{kawa}) is not $C^3$ smooth at zero, namely, there is no $T>0$
such that the solution map of:
$$
u_0\in H^s({\mathbb R})\mapsto u\in C([0,T]; H^s({\mathbb R}))
$$
is $C^3$ at zero.
\end{thm}
\begin{proof}
Under $s<-\frac 94$
for contradiction we assume that the solution map

$$
u_0\in H^s({\mathbb R})\mapsto u\in C([0,T]; H^s({\mathbb R}))
$$
is $C^3$ at zero. According to \cite{Bo97}, we must have
\begin{equation*}\label{normI}
\sup_{t\in
[0,T]}\|A_3(f)\|_{H^s}\lesssim\|f\|_{H^s}^3\quad\hbox{for\ all}\quad
f\in H^s(\mathbb R),
\end{equation*}
where
$$
A_3(f)(x,t)=\int_0^t W(t-\tau)\Big(\partial_x(A_1(f)\cdot
A_2(f))\Big)(\tau)d\tau;
$$
$$
A_2(f)(x,t)=\int_0^t
W(t-\tau)\Big(\partial_x(A_1(f)^2)\Big)(\tau)d\tau;
$$
$$
A_1(f)(x,t)=W(t)f=S_t \ast f(x),
$$
where $ \widehat{S}_t = e^{it \dr(\xi)}$ with $ \dr(\xi) =
\mu\xi^3-\xi^5$, or
$$
S_t(x)  = \int e^{i(x\xi + t \dr(\xi))}\,d\xi.
$$
Motivated by the selection of a test function in \cite{Bo97} and
\cite{Tz} we choose an $H^s(\mathbb R)$-function $f$ with
$$
\|f\|_{H^s}\sim
1\quad\hbox{and}\quad\hat{f}(\xi)=r^{-1/2}N^{-s}\chi_{[-r,r]}(|\xi|-N),
$$
where $r=(N^{3/2}\log N)^{-1}$, $N>0$ is sufficiently large, and
$\chi_{E}$ stands for the characteristic function of a set
$E\subseteq\mathbb R$.

The key issue is to control $\|A_3(f)\|_{H^s}$ from below. To
proceed, we make the following estimates:
$$
A_1(f)(x,t)\sim r^{-1/2}N^{-s}\int_{|\xi\pm
N|<r}e^{it\dr(\xi)+ix\xi}\,d\xi
$$
and
$$
A_2(f)(x,t)\sim F_1(x,t)-F_2(x,t)
$$
where
$$
F_1(x,t)=r^{-1}N^{-2s}\iint_{\max_{j=1,2}|\xi_j\pm
N|<r}\frac{(\xi_1+\xi_2)e^{ix(\xi_1+\xi_2)+it
(\dr(\xi_1)+\dr(\xi_2))}}{\dr(\xi_1)+\dr(\xi_2)-\dr(\xi_1+\xi_2)}\,d\xi_1d\xi_2
$$
and
$$
F_2(x,t)=r^{-1}N^{-2s}\iint_{\max_{j=1,2}|\xi_j\pm
N|<r}\frac{(\xi_1+\xi_2)e^{ix(\xi_1+\xi_2)+it\dr(\xi_1+\xi_2)}}
{\dr(\xi_1)+\dr(\xi_2)-\dr(\xi_1+\xi_2)}\,d\xi_1d\xi_2.
$$

The contribution of $F_1$ to $A_3(f)$ is comparable with
\begin{equation}\label{com}
r^{-3/2}N^{-3s}\iiint_{\max_{j=1,2,3}|\xi_j\pm N|<r}\frac
{e^{ix(\xi_1+\xi_2+\xi_3)+it\dr(\xi_1+\xi_2+\xi_3)}}{Q_1(\xi_1,\xi_2,\xi_3)^{-1}
Q_2(\xi_1,\xi_2,\xi_3)^{-1}}\,d\xi_1d\xi_2d\xi_3,
\end{equation}
where
$$
Q_1(\xi_1,\xi_2,\xi_3):=\frac{(\xi_1+\xi_2+\xi_3)
(\xi_2+\xi_3)}{\dr(\xi_2)+\dr(\xi_3)-\dr(\xi_2+\xi_3)}
$$
and
$$
Q_2(\xi_1,\xi_2,\xi_3):=\frac{e^{it(\dr(\xi_1)+\dr(\xi_2)+\dr(\xi_3)-
\dr(\xi_1+\xi_2+\xi_3))}-1}{
\dr(\xi_1)+\dr(\xi_2)+\dr(\xi_3)-\dr(\xi_1+\xi_2+\xi_3)}.
$$
Setting
$$
\theta=\dr(\xi_1)+\dr(\xi_2)+\dr(\xi_3)-\dr(\xi_1+\xi_2+\xi_3),
$$
we employ $\dr(\xi)=-\xi^5+\mu\xi^3$ to get
$$
\theta=5(\xi_1+\xi_2)(\xi_1+\xi_3)(\xi_2+\xi_3)\Big(\frac{\xi_1^2+\xi_2^2+\xi_3^2}2+
\frac{(\xi_1+\xi_2+\xi_3)^2}2-\frac{3\mu}{5}\Big).
$$
Thus
$$
|\theta|\sim N^5\quad\hbox{or}\quad |\theta|\lesssim r^2N^3\sim
(\log N)^{-2}.
$$
This tells us that the major contribution to (\ref{com}) is
obviously gotten from the second alternative, in which case we get
$$
G_1(x,t)=r^{-3/2}N^{-3s}\iiint_{\max_{j=1,2,3}\{|\xi_j\pm N|\}<r,\
|\theta|\lesssim r^2N^3}\frac{e^{i(x(\xi_1+\xi_2+\xi_3)+
t\dr(\xi_1+\xi_2+\xi_3))}}{Q_1(\xi_1,\xi_2,\xi_3)^{-1}}\,d\xi_1d\xi_2d\xi_3
$$
with
$$
\|G_1\|_{H^s}\sim r^{-1}N^{-3s}N^{-4}N^{1+s}r^2\sim
{N^{-2s-9/2}}{(\log N)^{-1}}.
$$

On the other hand, the contribution of $F_2$ to $A_3(f)$ is
comparable with
$$
G_2(x,t)=r^{-3/2}N^{-3s}\iiint_{\max_{j=1,2,3}|\xi_j\pm
N|<r}\frac{e^{i(x(\xi_1+\xi_2+\xi_3)+
t\dr(\xi_1+\xi_2+\xi_3))}}{Q_1(\xi_1,\xi_2,\xi_3)^{-1}
Q_3(\xi_1,\xi_2,\xi_3)^{-1}}\,d\xi_1d\xi_2d\xi_3,
$$
where
$$
Q_3(\xi_1,\xi_2,\xi_3):=\frac{e^{it(\dr(\xi_1)+\dr(\xi_2+\xi_3)-
\dr(\xi_1+\xi_2+\xi_3))}-1}{\dr(\xi_1)+\dr(\xi_2+\xi_3)-\dr(\xi_1+\xi_2+\xi_3)}
$$
and
\begin{eqnarray*}
\|G_2\|_{H^s}&\lesssim&
r^{-3/2}N^{-2s-5}\left\|\iiint_{\max_{j=1,2,3}|\xi_j\pm
N|<r}\frac{e^{ix(\xi_1+\xi_2+\xi_3)}}{|\xi_2+\xi_3|+N^{-4}}
\,d\xi_1d\xi_2d\xi_3\right\|_{L^2_{x}(\mathbb
R)}\\
&\lesssim& r^{-1}N^{-2s-5}\iint_{\max_{j=2,3}|\xi_j\pm N|<r}
(|\xi_2+\xi_3|+N^{-4})^{-1}\,d\xi_2d\xi_3\\
&\lesssim& N^{-2s-5}\log N.
\end{eqnarray*}
Consequently, we get
$$
\frac{N^{-2s-9/2}}{\log N}\left(1-\Big(\frac{\log
N}{N^{1/4}}\Big)^2\right)\lesssim\|G_1\|_{H^s}-\|G_2\|_{H^s}\lesssim
\|A_3(f)\|_{H^s}\lesssim 1
$$
whence deriving $s\ge -9/4$ (via letting $N\to\infty$) -- a
contradiction to $s<-9/4$. This completes the proof of Theorem
\ref{major2}.
\end{proof}

\end{document}